\documentclass[11pt]{article}

\usepackage[T1]{fontenc}
\usepackage[english]{babel}
\usepackage[utf8]{inputenc}
\usepackage{amsmath,amssymb,amsfonts,amsthm}
\usepackage{graphicx}
\usepackage{enumitem}
\usepackage{bbold}
\usepackage{tikz}
\usepackage{caption}
\usepackage{subcaption}
\usepackage{comment}
\usepackage{listings} 
\usepackage{circuitikz}
\usepackage{authblk}
\usepackage{pdfpages}
\usepackage[scale=0.75]{geometry}

\theoremstyle{plain}
\newtheorem{theorem}{Theorem}
\newtheorem{lemma}[theorem]{Lemma}

\theoremstyle{definition}
\theoremstyle{remark}

\newtheorem{remark}{Remark}

\usepackage{color}
\usepackage{xcolor}

\newcommand{\Qmark}{{\usefont{T1}{qcr}{m}{n}?}}
\newcommand{\ind}{\mathbb{1}} 

\newcommand{\NN}{\mathbb{N}}

\newcommand{\ZZ}{\mathbb{Z}}
\newcommand{\prob}[1]{\mathbb{P}\left( #1 \right)}
\newcommand{\esp}[1]{\mathbb{E}\left[ #1 \right]}

\newcommand{\env}[1]{\tilde{H}_{#1}}
\newcommand{\hard}[1]{H_{#1}}

\definecolor{codegreen}{rgb}{0,0.6,0}
\definecolor{codegray}{rgb}{0.5,0.5,0.5}
\definecolor{codepurple}{rgb}{0.58,0,0.82}
\definecolor{backcolour}{rgb}{0.94,0.94,0.94}

\lstdefinestyle{mystyle}{
    backgroundcolor=\color{backcolour},   
    commentstyle=\color{codegreen},
    keywordstyle=\color{magenta},
    numberstyle=\tiny\color{codegray},
    stringstyle=\color{codepurple},
    basicstyle=\ttfamily\footnotesize,
    breakatwhitespace=false,         
    breaklines=true,                 
    captionpos=b,                    
    keepspaces=true,                 
    numbers=left,                    
    numbersep=5pt,                  
    showspaces=false,                
    showstringspaces=false,
    showtabs=false,                  
    tabsize=2
}

\title{Ergodicity of the hard-core PCA \\ with a random walk method}
\author[1]{J\'er\^ome Casse}
\author[2]{Ir\`ene Marcovici}
\author[2]{Maxence Poutrel}
\affil[1]{\small Université Paris-Saclay, CNRS,
Laboratoire de mathématiques d’Orsay, 91405 Orsay, France\\ \texttt{jerome.casse@universite-paris-saclay.fr}}
\affil[2]{\small Univ Rouen Normandie, CNRS, Normandie Univ, LMRS UMR 6085, F-76000 Rouen, France\newline \texttt{\{irene.marcovici,maxence.poutrel1\}@univ-rouen.fr}}

\begin{document}

\maketitle

\begin{abstract}
  The hard-core probabilistic cellular automaton has attracted a renewed interest in the last few years, thanks to its connection with the study of a combinatorial game on percolation configurations. We provide an alternative proof for the ergodicity of this PCA for a neighbourhood of size $2$ and $3$, using the notion of decorrelated islands introduced by Casse in 2023, together with some new ideas. This shortens the previous proofs and provides a more intuitive and unified approach.
\end{abstract}

\section{Percolation games and PCA}

\paragraph{Percolation games.}
Let us present the following percolation game on~$\ZZ^2$, as in~\cite{HMM19}. Consider two non-negative reals $\epsilon_0,\epsilon_1 \in [0,1]$ such that $0 \leq \epsilon_0 + \epsilon_1 \leq 1$. For each site of $\ZZ^2$ we assign independently one of the 3 following states:
\begin{itemize}[label=$\bullet$]
	\item trap with probability $\epsilon_1$;
	\item target with probability $\epsilon_0$;
	\item open with probability $r=1-\epsilon_0-\epsilon_1$.
\end{itemize}
This defines the random board on which two players compete. 
Fix an integer $n \geq 2$. The rule of the game is the following. At time $0$, a token is placed at the starting position $(0,0)$, and then, the two players move it alternatively, from its current position $(i,j)$ on the board to a site in the set $\text{Out}(i,j)=\{(i+k,j+1) \, : \, k \in \{0,1,\dots,n-1\}\},$ see~Figure~\ref{fig:moves}. If the current player moves the token to a trap, that player loses the game immediately; if it moves it to a target, that player wins the game immediately; otherwise (i.e.\ if the destination site is open), the game continues with the other player's turn. 

\begin{figure}
  \begin{center}
    
    \begin{tikzpicture}[scale=0.8]
      \draw[step=1cm,color=black,ultra thin] (-.4,-.4) grid (4.4,3.4);
      \node[fill=black,inner sep=2pt,shape=circle,label={[xshift=.5cm, yshift=-.7cm]\footnotesize$(i,j)$}] at (1,1){};
      \node[fill=black,inner sep=2pt,shape=circle,label={[xshift=-.5cm]\footnotesize$(i,j+1)$}] at (1,2){};
      \node[fill=black,inner sep=2pt,shape=circle,label={[xshift=.5cm]\footnotesize$(i+1,j+1)$}] at (2,2){};
      \draw[ultra thick,color=black] (1,1) -- (1,2) node[
      currarrow,
      pos=0.5, 
      xscale=1,
      sloped,
      scale=1] {};
      \draw[ultra thick,color=black] (2,2) -- (1,1) node[
      currarrow,
      pos=0.5, 
      xscale=1,
      sloped,
      scale=1] {};
    \end{tikzpicture}\hspace{10mm}
    \begin{tikzpicture}[scale=0.8]
      \draw[step=1cm,color=black,ultra thin] (-.4,-.4) grid (4.4,3.4);
      \node[fill=black,inner sep=2pt,shape=circle,label={[xshift=.5cm, yshift=-.7cm]\footnotesize$(i,j)$}] at (1,1){};
      \node[fill=black,inner sep=2pt,shape=circle,label={[xshift=-.5cm]\footnotesize$(i,j+1)$}] at (1,2){};
      \node[fill=black,inner sep=2pt,shape=circle,label={[xshift=.5cm]\footnotesize$(i+1,j+1)$}] at (2,2){};
      \node[fill=black,inner sep=2pt,shape=circle,label={[xshift=.5cm,yshift=-.8cm]\footnotesize$(i+2,j+1)$}] at (3,2){};
      \draw[ultra thick,color=black] (1,1) -- (1,2) node[
      currarrow,
      pos=0.5, 
      xscale=1,
      sloped,
      scale=1] {};
      \draw[ultra thick,color=black] (2,2) -- (1,1) node[
      currarrow,
      pos=0.5, 
      xscale=1,
      sloped,
      scale=1] {};
      \draw[ultra thick,color=black] (3,2) -- (1,1) node[
      currarrow,
      pos=0.5, 
      xscale=1,
      sloped,
      scale=1] {};
    \end{tikzpicture}
  \end{center}
  \caption{The possible moves in the percolation game, for $n=2$ on the left, and $n=3$ on the right} \label{fig:moves}
\end{figure}
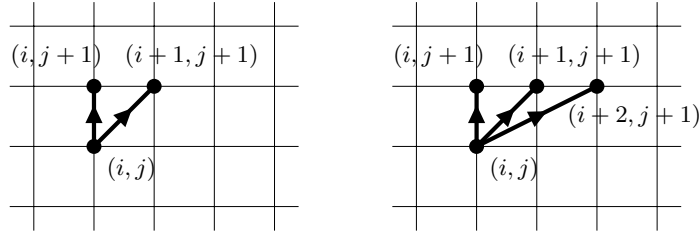

We can then ask the following question: with the best strategy for each player, what is the probability that the game never ends, so that there is a draw? The answer is $0$ for the cases $n=2$~\cite{HMM19} and $n=3$~\cite{BKPR22} when $0< \epsilon_0 + \epsilon_1 \leq 1$. In this article, we give alternative proofs of these two results, using ideas developed in~\cite{Casse23}. For $n \geq 4$, the question remains open.

\paragraph{Associated PCA.} To study the percolation game, we can associate to each site the outcome (win, lose or draw) of the current player (the one who moves out the token from this site) when both players play perfectly. If we know the outcomes of the sites of the set $\text{Out}(i,j)$, it allows to find out the outcome of site $(i,j)$. This leads to the introduction of a probabilistic cellular automaton (PCA) having a neighborhood of size $n$, with three states: $0$~for the win, $1$~for the lose and~\Qmark~for the draw. Note that a trap can be interpreted as a win site, since if a player moves the token to a trap, the next player (the one who should have moved the token out from the trap site) wins; while a target can be interpreted as a lose site.  

This PCA, denoted $\env{n}$ later in this article, defines a Markov chain $(X_i(t) : i \in \ZZ)_{t \geq 0}$ on $\{0,1,\text{\Qmark}\}^\ZZ$, whose transition probabilities satisfy:
\begin{displaymath}
  \prob{X_i(t+1) = a~|~\forall k \in \{0,\dots,n-1\},\, X_{i+k}(t) = 0} =
  \begin{cases}
    \epsilon_0 & \text{if } a=0;\\
    1-\epsilon_0 & \text{if } a=1;\\
    0 & \text{if } a = \text{\Qmark}.
  \end{cases}
\end{displaymath}

\begin{displaymath}
  \prob{X_i(t+1)=a~|~\exists k \in \{0,\dots,n-1\},\, X_{i+k}(t)=1} =
  \begin{cases}
    1-\epsilon_1 &  \text{if } a=0;\\
    \epsilon_1&  \text{if } a=1;\\
    0&  \text{if } a=\text{\Qmark}.
  \end{cases}
\end{displaymath}

\begin{displaymath}
  \prob{X_i(t+1)=a~|~(X_i(t),\ldots,X_{i+n-1}(t))\in\{0,\text{\Qmark}\}^n\setminus\{(0,\ldots,0)\}} =
  \begin{cases}
    \epsilon_0 &  \text{if } a=0;\\
    \epsilon_1&  \text{if } a=1;\\
    r &  \text{if } a=\text{\Qmark}.
  \end{cases}
\end{displaymath}

These transitions are illustrated on Figure~\ref{fig:trans2} for $n=2$.

\begin{figure}
  \begin{center}
    \begin{tikzpicture}[scale=0.8]
      \foreach \i in {0,1}{
        \node [fill=none,inner sep=4pt,shape=circle] (voisin\i) at (\i*1.5,0) {0};    
      }
      \node[fill=none,inner sep=4pt,shape=circle] (cell0) at (0,-2) [label=right: w.p.\  $\epsilon_0$] {0}; 
      \node [fill=none,inner sep=4pt,shape=circle] at (0,-2.7) [label=right: w.p.\  $1-\epsilon_0$] {1};
      \node [fill=none,inner sep=4pt,shape=circle] at (0,-3.6) {};
      \draw [-] (voisin0.south) -- (cell0.north);
      \draw [-] (voisin1.south) -- (cell0.north);
    \end{tikzpicture}\hspace{1.5cm}
    \begin{tikzpicture}[scale=0.8]
      \node [fill=none,inner sep=4pt,shape=circle] (voisin0) at (0,.75) {1};    
      \node [fill=none,inner sep=4pt,shape=circle] (voisin1) at (1.5,.75) {\textbullet};
      \node [fill=none,inner sep=4pt,shape=circle] (voisin0) at (0,0) {\textbullet};    
      \node [fill=none,inner sep=4pt,shape=circle] (voisin1) at (1.5,0) {1};
      \node[fill=none,inner sep=4pt,shape=circle] (cell0) at (0,-2) [label=right: w.p.\ $1-\epsilon_1$] {0}; 
      \node [fill=none,inner sep=4pt,shape=circle] at (0,-2.7) [label=right: w.p.\ $\epsilon_1$] {1};
      \node [fill=none,inner sep=4pt,shape=circle] at (0,-3.6) {};
      \draw [-] (voisin0.south) -- (cell0.north);
      \draw [-] (voisin1.south) -- (cell0.north);
    \end{tikzpicture}\hspace{1.5cm}
    \begin{tikzpicture}[scale=0.8]
      \node [fill=none,inner sep=4pt,shape=circle] at (0,1.5) {\text{\Qmark}};    
      \node [fill=none,inner sep=4pt,shape=circle] at (1.5,1.5) {\text{\Qmark}};
      \node [fill=none,inner sep=4pt,shape=circle] at (0,.75) {\text{\Qmark}};    
      \node [fill=none,inner sep=4pt,shape=circle] at (1.5,.75) {0};
      \node [fill=none,inner sep=4pt,shape=circle] (voisin0) at (0,0) {0};    
      \node [fill=none,inner sep=4pt,shape=circle] (voisin1) at (1.5,0) {\text{\Qmark}};
      \node[fill=none,inner sep=4pt,shape=circle] (cell0) at (0,-2) [label=right: w.p.\ $\epsilon_0$] {0}; 
      \node [fill=none,inner sep=4pt,shape=circle] at (0,-3.4) [label=right: w.p.\ ${r = 1-\epsilon_0-\epsilon_1}$] {\text{\Qmark}};
      \node [fill=none,inner sep=4pt,shape=circle] at (0,-2.7) [label=right: w.p.\ $\epsilon_1$] {1};
      \draw [-] (voisin0.south) -- (cell0.north);
      \draw [-] (voisin1.south) -- (cell0.north);
    \end{tikzpicture}
  \end{center}
  \caption{Transition probability of the PCA $\env{2}$.} \label{fig:trans2}
\end{figure}
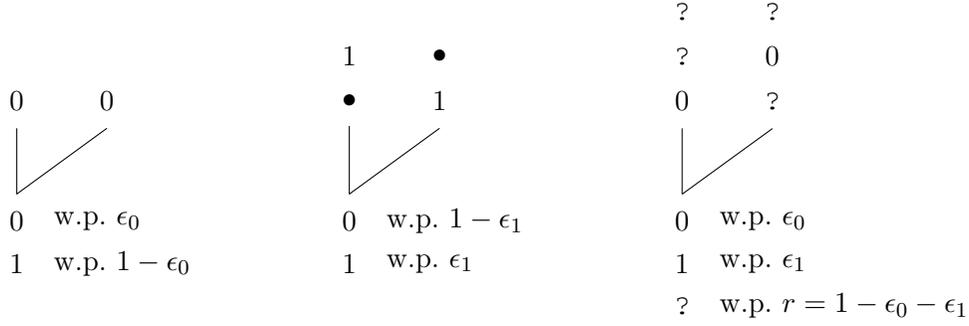

A PCA is \emph{ergodic} if there is a unique probability measure $\mu_\infty$ such that the Markov chain $Z(t)$ of the PCA verifies: for any $\mu_0\in\{0,1\}^\ZZ$ such that the initial configuration $Z(0)$ is $\mu_0$, the law of $Z(t)$ weakly converges to $\mu_\infty$. Here, the ergodicity of $\env{n}$ means that the status of site $(0,0)$ does not depend on the assignment of traps and targets that are sufficiently far away on the board. Observe that if for all $i\in \ZZ$, $X_i(0) \in \{0,1\}$, then for all $i \in \ZZ$, $t \geq 0$, $X_i(t) \neq \; \text{\Qmark}$. In particular, if the PCA $\env{n}$ is ergodic, then its unique invariant measure does not contain the state $\text{\Qmark}$, meaning that the probability of draw is null.

\begin{theorem}\label{thm:main}
  For $n\in\{2,3\}$ and $(\epsilon_0,\epsilon_1)\in[0,\frac{1}{2}]^2\setminus\{(0,0)\}$, the PCA $\env{n}$ is ergodic, and so there is almost surely no draw in the percolation game.
\end{theorem}

The result above was first proved for $n=2$ using weight functions~\cite{HMM19}. The method was then adapted  for $n=3$, with highly intricate computations~\cite{BKPR22}. The aim of the present article is to provide another shorter proof, based on the method of decorrelated islands~\cite{Casse23}.

Note that the PCA $\env{n}$ can be seen as the \emph{envelope PCA} of the \emph{hard-core PCA} $\hard{n}$. By definition, it is the binary PCA whose associated Markov chain $X(t)=(X_i(t): i \in \ZZ) \in \{0,1\}^\ZZ$ satisfies:
\begin{displaymath}
  \prob{X_i(t+1) = a~|~\forall k \in \{0,\dots,n-1\}\, X_{i+k}(t) = 0} =
  \begin{cases}
    \epsilon_0 & \text{if } a=0;\\
    1-\epsilon_0 & \text{if } a=1.
  \end{cases}
\end{displaymath}

\begin{displaymath}
  \prob{X_i(t+1)=a~|~\exists k \in \{0,\dots,n-1\}\, X_{i+k}(t)=1} =
  \begin{cases}
    1-\epsilon_1 &  \text{if } a=0;\\
    \epsilon_1&  \text{if } a=1.
  \end{cases}
\end{displaymath}
The transitions are illustrated on Figure~\ref{fig:trans3} for $n=3$. The PCA $\hard{n}$ corresponds to the restriction of $\env{n}$ to configurations that do not contain the state $\text{\Qmark}$. Envelope PCA  are a practical tool to prove the ergodicity of PCA in the high noise regime, using the fact that the ergodicity of the envelope PCA implies the ergodicity of the associated PCA~\cite{BMM13}. 

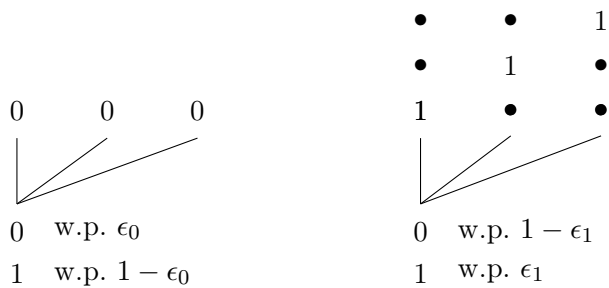
\begin{figure}
  \begin{center}
  \begin{tikzpicture}[scale=0.8]
    \foreach \i in {0,1,2}{
      \node [fill=none,inner sep=4pt,shape=circle] (voisin\i) at (\i*1.5cm,0) {0};    
    }
    \node[fill=none,inner sep=4pt,shape=circle] (cell0) at (0,-2) [label=right: w.p.\ $\epsilon_0$] {0}; 
    \node [fill=none,inner sep=4pt,shape=circle] at (0,-2.7) [label=right: w.p.\ $1-\epsilon_0$] {1};
    \draw [-] (voisin0.south) -- (cell0.north);
    \draw [-] (voisin1.south) -- (cell0.north);
    \draw [-] (voisin2.south) -- (cell0.north);
  \end{tikzpicture} \hspace{2cm}
  \begin{tikzpicture}[scale=0.8]
    \node [fill=none,inner sep=4pt,shape=circle] (voisin0) at (0,0) {1};    
    \node [fill=none,inner sep=4pt,shape=circle] (voisin1) at (1.5,0) {\textbullet};
    \node [fill=none,inner sep=4pt,shape=circle] (voisin1) at (3,0) {\textbullet};
    \node [fill=none,inner sep=4pt,shape=circle] (voisin0) at (0,.75) {\textbullet};    
    \node [fill=none,inner sep=4pt,shape=circle] (voisin1) at (1.5,.75) {1};
    \node [fill=none,inner sep=4pt,shape=circle] (voisin1) at (3,.75) {\textbullet};
    \node [fill=none,inner sep=4pt,shape=circle] (voisin0) at (0,1.5) {\textbullet};    
    \node [fill=none,inner sep=4pt,shape=circle] (voisin1) at (1.5,1.5) {\textbullet};
    \node [fill=none,inner sep=4pt,shape=circle] (voisin1) at (3,1.5) {1};
    \node [fill=none,inner sep=4pt,shape=circle] (voisin0) at (0,0) {1};    
    \node [fill=none,inner sep=4pt,shape=circle] (voisin1) at (1.5,0) {\textbullet};
    \node [fill=none,inner sep=4pt,shape=circle] (voisin2) at (3,0) {\textbullet};
    \node[fill=none,inner sep=4pt,shape=circle] (cell0) at (0,-2) [label=right: w.p.\  $1-\epsilon_1$] {0}; 
    \node [fill=none,inner sep=4pt,shape=circle] at (0,-2.7) [label=right: w.p.\  $\epsilon_1$] {1};
    \draw [-] (voisin0.south) -- (cell0.north);
    \draw [-] (voisin1.south) -- (cell0.north);            
    \draw [-] (voisin2.south) -- (cell0.north);
  \end{tikzpicture}
\end{center}
\caption{Transition probability of the hard-core PCA $H_3$, where dots mean that there can be either a 0 or a 1} \label{fig:trans3}
\end{figure}

Finally, let us also mention that $H_n$ can be seen as a CA with \emph{double errors}: it is a PCA obtained from a deterministic CA, on which we add an error depending on the expected state ($\epsilon_0$ if we expect a $1$ and $\epsilon_1$ if we expect a $0$).

\paragraph{Ergodicity of PCA.} The ergodicity of PCA has been studied using various methods (coupling, entropy, Fourier analysis, weight functions \ldots) and still raises many questions, see~\cite{TOOM90,MM14b,MST19} and the references therein. The hard-core PCA has attracted particular interest, since in addition to its connection to percolation games, it has long been known for its connection with statistical physics and with the enumeration of directed animals. For $\epsilon_1=0$, $\hard{n}$ has a Markovian invariant measure and its ergodicity is easier to prove~\cite{HMM19}, but for other values of the parameters, no general method is known.

\paragraph{Outline of the article.} In Section~\ref{sec:Heuristic}, we present the sketch of our proof of Theorem~\ref{thm:main}. In Section~\ref{sec:di}, we recall the notion of decorrelated islands introduced in~\cite{Casse23}, then in Section~\ref{sec:ni}, we explain the new idea that allows to handle the case of the hard-core PCA. In Section~\ref{sec:proof2}, we prove Theorem~\ref{thm:main} for $n=2$. This requires some computations that we present in details. In Section~\ref{sec:proof3}, we prove Theorem~\ref{thm:main} when $n=3$, using the same method. 

\section{Sketch of the proof}\label{sec:Heuristic}

\subsection{Decorrelated islands method} \label{sec:di}

Let us consider the evolution of the PCA $\env{n}$, from the initial configuration $\text{\Qmark}^\ZZ$. 
\begin{itemize}[label=$\bullet$]
    \item With probability $(\epsilon_0+\epsilon_1)^k$, the $k$ consecutive cells of indices $0,\dots,k-1$ go from state~$\text{\Qmark}$ to state $0$ or $1$, i.e.\ for all $i \in \{0,\dots,k-1\}$, $X_i(1) \in \{0,1\}$. We call such a  sequence of consecutive cells a \emph{decorrelated island}, as the states of these cells do not depend on the initial configuration.
    \item Our goal is to study the evolution of the size of a decorrelated island, that is a (maximal) set of consecutive cells where the symbol~$\text{\Qmark}$ does not appear. With probability~$1$, there exists a time $t_0\geq 0$ when such a decorrelated island is created. For $t\geq t_0$, we denote by $i_t$ and $j_t$ the positions of the left and right boundary of the decorrelated island, respectively. So, cells between $i_t$ and $j_t$ are in state $0$ or $1$, while $i_t-1$ and $j_t+1$ are in state \Qmark.
    \item If $j_t-i_t\to_{t\to\infty}\infty$ with positive probability, this means that the island grows, so that the symbols~\Qmark~progressively disappear, implying that the PCA is ergodic.
\end{itemize}

To study the behaviour of the left boundary $i_t$, we only remember the states of the first $m$ cells of the island, given by $f_t = (X_{i_t},X_{1+i_t},\dots,X_{m-1+i_t})$, with a well-chosen value of $m$, selected according to the size of the neighborhood $n$. We proceed in a similar way for the right boundary $j_t$. In the following, we will set $m=2$ when $n=2$, and $m=4$ when $n=3$.

This method, applied with $m=1$, gives a sufficient ergodicity criterion for PCA with binary alphabet and two-size neighborhood~\cite{Casse23}. This allows to handle the case of $12$ out of $16$ deterministic cellular automata with error $\epsilon > 0$. By taking $m=2$, it also gives a new proof of the ergodicity of the PCA $\env{2}$ (and so of $\hard{2}$) when $\epsilon_0 = \epsilon_1 > 0$.

\subsection{The main improvement of the method} \label{sec:ni}

In this article, we improve the method developed in~\cite{Casse23} to give a new and shorter proof of Theorem~\ref{thm:main}. The improvement consists in considering one or two time steps at once, according to the value of the states of the boundary, and not just one time step in every situation.

First, observe that in order to prove that $\prob{j_t-i_t\to\infty}$ is positive, since the size of negative steps is bounded (by $1$), it is enough to show that the mean asymptotic drift
\begin{displaymath}
  D = \lim_{t\to\infty} \esp{J_t} - \lim_{t\to\infty}\esp{I_t}
\end{displaymath}
of $(j_t-i_t)_{t\geq t_0}$ is positive, where $J_t$ is the increment of $(j_{t})$, i.e.\ $J_t = j_{t+1}-j_t$ and $I_t$ the one of $(i_{t})$. The following remark in fact allows one to restrict ourselves to the study of the right boundary.

\begin{remark}\label{rem:symm}
Let us denote by $R =\lim_{t\to\infty}\esp{J_t}$ the mean asymptotic drift of the right boundary. Using the symmetry of the transition rule of $\env{n}$ with respect to the left-right symmetry, while taking into account the fact that the neighborhood is not centered, it follows that $D=2R+(n-1)$. Consequently, in the rest of the article, we focus only on the behaviour of the right boundary.
\end{remark}

The innovation of the present work is to study the law of $J_t$ and the one of $J_{t}+J_{t+1}$ according to the values of the states of the boundary $f_t$. This is summed up in the following lemma.

\begin{lemma}\label{LemMinMax}
Consider a homogeneous Markov chain $(j_t,f_t)$ on $\ZZ \times F$, where $F$ is a finite set, and assume that the increment $J_t = j_{t+1}-j_t$ depends only on $f_t$. Precisely, assume that for any $f \in F$, for any $k\in\ZZ$, and for any $t\geq 0$, $$\esp{J_t| f_t =f, j_t=k } = \esp{j_1|f_0=f,j_0=0}.$$ Then, if the limit $R = \lim_{t \to \infty} \esp{J_t}$ exists and is finite, it satisfies:
  \begin{align}\label{eq:rightminmax}
    R \geq \min_{f \in F} \max \left( \esp{J_t | f_t = f},\frac{1}{2} \esp{J_t+J_{t+1} | f_t = f} \right).
  \end{align}
\end{lemma}

\begin{proof}
First, we split the set $F$ into two disjoint parts:
\begin{displaymath}
 F_1 = \left\{f \in F : \esp{J_0 | f_0 = f} \geq \frac{1}{2} \esp{J_0+J_1 | f_0 = f} \right\} \text{ and } F_2 = F \backslash F_1.  
\end{displaymath}
The set $F_1$ represents the states for which the mean increment is larger when we consider a single time step compared to two time steps, and $F_2$ those for which it is the reverse.

We now define another Markov chain $(\hat{j}_t,\hat{f}_t,w_t)_{t\geq0}$ as follows:
\begin{itemize}[label=$\bullet$]
    \item $(\hat{j}_0,\hat{f}_0,w_0) = (j_0,f_0,\ind_{f_0 \in F_2})$,
    \item if $w_t = 0$, then $\hat{j}_{t+1} = j_{t+1}$, $\hat{f}_{t+1} = f_{t+1}$ and $w_{t+1} = \ind_{f_{t+1} \in F_2}$,
    \item if $w_{t} =1$, then  $\hat{j}_{t+1} = j_{t}$, $\hat{f}_{t+1} = f_t$ and $w_{t+1} = 0$.
\end{itemize}

As a consequence:
\begin{align*}
    \hat{J}_t = \hat{j}_{t+1} - \hat{j}_t = J_t \text{ when } w_t=0 \text{ and } f_t\in F_1;\\
    \hat{J}_t = 0 \text{ when } w_t=1 \text{ and } f_t\in F_2;\\
    \hat{J}_t = J_{t-1} + J_t \text{ when } w_t=0 \text{ and } f_t\in F_2.
\end{align*}
So, compared to $(j_t,f_t)_{t\geq 0}$, the new Markov chain moves one step forward when $\hat{f}_t$ is in $F_1$, while when it is in $F_2$, the state does not change at the first time step (when $w_t=1$) and then moves of the equivalent of two steps at a time (when $w_t=0$). The variable $w_t$ can thus be thought as a waiting time. \\

Let us now consider the value $\hat{R} = \lim_{t\to\infty} \esp{\hat{J}_t}$.
We show that $\hat{R}\geq M$, where $M$ is the term on the right-hand side of (\ref{eq:rightminmax}).
We have 
\begin{displaymath}
    \hat{R} = \sum_{f \in F_1} \tau(f,0) \esp{\hat{J}_t | \hat{f}_t = f, w_t=0} + \sum_{f \in F_2}  \left( \tau(f,1) \times 0 + \tau(f,0) \esp{\hat{J}_t | \hat{f}_t = f, w_t=0}\right)
\end{displaymath}
where $\tau$ is the invariant measure of the Markov chain $(\hat{f}_t,w_t)_{t\geq 0}$, which exists since $F$ is finite.

By definition of $(\hat{j}_t,\hat{f}_t,w_t)_{t\geq 0}$, we have for any $f\in F_1$,
\begin{align*}
    \esp{\hat{J}_t | \hat{f}_t = f, w_t=0} = \max \left( \esp{J_t | f_t = f},\frac{1}{2} \esp{J_t + J_{t+1} | f_t = f} \right) \geq M
\end{align*}
and for any $f \in F_2$, 
\begin{align*}
     \esp{\hat{J}_{t} | \hat{f}_t = f, w_t=0} = 2\max \left( \esp{J_t | f_t = f}, \frac{1}{2} \esp{J_t + J_{t+1} | f_t = f} \right) \geq 2M.
\end{align*} 
By construction of the new Markov chain, for any $f \in F_2$, $\tau(f,0) = \tau(f,1)$. So we obtain 
$\hat{R} \geq M \sum_{(f,w) \in (F_1\times\{0\}) \cup (F_2 \times \{0,1\})} \tau(f,w) = M$.

To conclude, we only need to prove that $R = \hat{R}$. We recall that the equality $(j_t,f_t) = (\hat{j}_t,\hat{f}_t)$ holds at least once in two, and by the ergodic theorem,
\begin{align*}
    \hat{R} = \lim_{t \to \infty} \esp{\hat{J}_t} = \lim_{t \to \infty} \frac{\esp{\hat{j}_t-j_0}}{t} \text{ and } R = \lim_{t \to \infty} \esp{J_t} = \lim_{t \to \infty} \frac{\esp{j_t-j_0}}{t}.
\end{align*}
So, we have $\hat{R} = R$. \qed
\end{proof}
Using Remark~\ref{rem:symm} together with Lemma~\ref{LemMinMax}, we can then prove Theorem~\ref{thm:main} by just showing that, for any $n \in \{2,3\}$ and for every $t\geq t_0$,
\begin{equation} \label{eq:goal}
  R \geq \min_{f \in \{0,1\}^m}\left( \max \left( \esp{J_t|f_t = f},\frac{1}{2}\esp{J_{t}+J_{t+1} |f_t = f} \right) \right) > -\frac{n-1}{2}.
\end{equation}

\begin{remark}\label{rem:modified}
We will actually prove the previous inequality not for $(j_t)$ but for some modification $(\tilde j_t)$, introduced in Sections~\ref{sec:mod2} for $n=2$ and in~\ref{sec:mod3} for $n=3$.
\end{remark}

\section{Proof of Theorem~\ref{thm:main} when $n=2$} \label{sec:proof2}

\subsection{Modified boundary and star state}\label{sec:mod2}

To give a formal proof of Theorem~\ref{thm:main}, we now introduce two additional notions.

First, as mentioned in Remark~\ref{rem:modified}, we define a slight modification of the position $j_t$, that depends on the states of the boundaries. To do so, let us consider a decorrelated island $(X_{i_t},\dots,X_{j_t})$ not reduced to a singleton at time $t$. To study the behaviour of its right boundary, we introduce the modified position $\tilde{j}_t$ defined by

\begin{displaymath}
  \Tilde{j}_{t} = \begin{cases}
    j_t & \text{if } (X_{-1+j_t},X_{j_t}) \in\{(0,1),(1,1)\};\\
    j_t-\frac{1}{2} & \text{if } (X_{-1+j_t},X_{j_t})=(0,0);  \\
    j_t-1 & \text{if } (X_{-1+j_t},X_{j_t}) = (1,0).
  \end{cases}
\end{displaymath}
As a result of this technical change, if $\epsilon_0=\epsilon_1=0$ (error-free regime), then the deterministic drift satisfies $\tilde{J}_t = -1/2$ for any $t$, whereas $J_t$ oscillates between $0$ and $-1$. The purpose of this modification is therefore only to simplify the analysis. Note that the cells we will consider at time $t$ in the right boundary will still be given by $f_t = (X_{-1+j_t},X_{j_t})$. The modification $\Tilde{i}_{t}$ of $i_t$ is defined analogously, using the left-right symmetry of the transition rule.

Second, we introduce a star state, that we denote by $*$. This state will be used to encode the state of a cell of the decorrelated island whose exact value ($0$ or $1$) is not remembered. Indeed, if we retain all the information about the island, the computation is too complex. The state $*$ will thus enable us to establish some bounds for the behaviour of the boundaries. This is different from the state $\text{\Qmark}$, which represents a cell whose value is completely unknown.

\subsection{Transition probabilities}

In this section, we study the behaviour of the Markov chain $(\Tilde{j}_{t},f_t)$ where we recall that $f_t = (X_{-1+j_t},X_{j_t})$. Our goal is to prove that
\begin{equation}
 \esp{ \Tilde{J}_t  | f_t \in \{(0,1),(1,0),(1,1)\} } > - \frac{1}{2} \text{ and } \esp{\Tilde{J}_t + \Tilde{J}_{t+1} | f_t = (0,0) } > - 1,
\end{equation}
which implies Equation~\eqref{eq:goal} when $n=2$, and so Theorem~\ref{thm:main} for $n=2$.

In the following, we assume that $j_t-i_t \geq 5$. It is not a loss of generality to achieve the result, and it allows to avoid any problem of dependency between $(\Tilde{i}_{t+1},(X_{i_{t+1}},X_{1+i_{t+1}}))$ and $(\Tilde{j}_{t+1},(X_{-1+j_{t+1}},X_{j_{t+1}}))$.

\subsubsection{Case $f_t \in \{(0,1),(1,0),(1,1)\}$:} For these three types of right boundaries, the transition probabilities are the same, as illustrated on Figure~\ref{fig:F1}. 
 
More precisely, if $\Tilde{j}_{t} = \Tilde{j}$ and if $f_t=f$ belongs to $\{(0,1),(1,0),(1,1)\}$, then
\begin{equation}
  (\Tilde{j}_{t+1},f_{t+1}) =
  \begin{cases}
    (\Tilde{j}-1,(1,0)) & \text{w.p.\ } \epsilon_1(1-\epsilon_1)r \\
    (\Tilde{j}-\frac{1}{2},(0,0)) & \text{w.p.\ } (1-\epsilon_1)^2r \\
    (\Tilde{j},(0,1)) & \text{w.p.\ } (1-\epsilon_1) \epsilon_1 r \\
    (\Tilde{j},(1,1)) & \text{w.p.\ } \epsilon_1^2r \\
    (\Tilde{j},(1,0)) & \text{w.p.\ } \epsilon_1 \epsilon_0 r \\
    (\Tilde{j}+\frac{1}{2},(0,0)) & \text{w.p.\ } (1-\epsilon_1) \epsilon_0 r \\
    (\Tilde{j}+1,(0,1)) & \text{w.p.\ } (1-\epsilon_1) \epsilon_1 r \\
    (\Tilde{j}+1,(1,1)) & \text{w.p.\ } \epsilon_1^2r \\
    (\Tilde{j}+k+1,(1,0)) & \text{w.p.\ } (\epsilon_0+\epsilon_1)^k \epsilon_1 \epsilon_0 r \text{ for } k \geq 0\\
    (\Tilde{j}+k+\frac{3}{2},(0,0)) & \text{w.p.\ } (\epsilon_0+\epsilon_1)^k \epsilon_0^2 r \text{ for } k \geq 0 \\
    (\Tilde{j}+k+2,(0,1)) & \text{w.p.\ } (\epsilon_0+\epsilon_1)^k \epsilon_0 \epsilon_1 r \text{ for } k \geq 0 \\
    (\Tilde{j}+k+2,(1,1)) & \text{w.p.\ } (\epsilon_0+\epsilon_1)^k \epsilon_1^2 r \text{ for } k \geq 0
  \end{cases} \label{eq:trans1}
\end{equation}

\begin{figure}
  \centering
  \scalebox{0.7}{%
    \begin{tikzpicture}[scale= 0.8]

        \begin{scope}[xshift=-2*2.5 cm,yshift={-0.5 cm}]
            \draw(0:0) ++ (45:1) -- ++ (0:-0.95cm);
        \end{scope}
        
        \begin{scope}[xshift=2.5 cm,yshift={-0.5 cm}]
            \draw (0:0) ++ (45:1) -- ++ (0:1.95cm);
            \draw (0:0) ++ (45:1) -- ++ (90:1.45cm);
            \draw (0:0) ++ (45:1) -- ++ (37:3.05cm);
        \end{scope}
        \begin{scope}[xshift=5 cm,yshift={-0.5 cm}]
            \draw (0:0) ++ (45:1) -- ++ (0:1.95cm);
            \draw (0:0) ++ (45:1) -- ++ (90:1.45cm);
            \draw (0:0) ++ (45:1) -- ++ (37:3.05cm);
        \end{scope}
        \begin{scope}[xshift=7.4 cm,yshift={-0.5 cm}]
            \draw(0:0) ++ (45:1) -- ++ (0:0.95cm);
            \draw (0:0) ++ (45:1) -- ++ (37:1.3cm);
        \end{scope}

        \foreach \i in {-2}{
				\begin{scope}[xshift=\i*2.5 cm,yshift={-0.5 cm}]
                    \draw (0:0) ++ (45:1) -- ++ (0:1.95cm);
                    \draw (0:0) ++ (45:1) -- ++ (90:1.45cm);
                    \draw (0:0) ++ (45:1) -- ++ (37:3.05cm);
                    \node [fill=white,inner sep=5.5pt,shape=circle,draw] at (45:1cm) {};
                    \node [fill=white,inner sep=6.5pt,shape=circle,draw] at (45:1cm)  {$*$};
				\end{scope}
                }
                
        \foreach \i in {-1,0}{
				\begin{scope}[xshift=\i*2.5 cm,yshift={-0.5 cm}]
                    \draw (0:0) ++ (45:1) -- ++ (0:1.95cm);
                    \draw (0:0) ++ (45:1) -- ++ (90:1.45cm);
                    \draw (0:0) ++ (45:1) -- ++ (37:3.05cm);
                    \draw (0,-2.5) rectangle +(1.4,4) [fill=white] ;
                    \node [fill=white,inner sep=5.5pt,shape=circle,draw] at (45:1cm){};
                    \node [fill=white,inner sep=6pt,shape=circle,draw] at (45:1cm) [label=below: $1-\epsilon_1$] {$0$};
				\end{scope}
                \begin{scope}[xshift=\i*2.5 cm,yshift={-2.5 cm}]
                    \node [fill=white,inner sep=6pt,shape=circle,draw] at (45:1cm) [label=below: $\epsilon_1$] {$1$};
                \end{scope}
				}
        \foreach \i in {3}{
            \begin{scope}[xshift=\i*2.5 cm,yshift={-0.5 cm}]
                \draw (0:0) ++ (45:1) -- ++ (90:1.45cm);
            \end{scope}
            \begin{scope}[xshift=\i*2.5 cm,yshift={1.5 cm}]
                \draw (0:0) ++ (45:1) -- ++ (0:-1.95cm);
            \end{scope}
        }
        \foreach \i in {1,...,3}{
            
            \begin{scope}[xshift=\i*2.5 cm,yshift={-0.5 cm}]
            \draw (0,-4.5) rectangle +(1.4,6) [fill=white] ;
            \node [fill=white,inner sep=6pt,shape=circle,draw] at (45:1cm) [label=below: $\epsilon_0$]{$0$};
            \end{scope}
            \begin{scope}[xshift=\i*2.5 cm,yshift={-2.5 cm}]
            \node [fill=white,inner sep=6pt,shape=circle,draw] at (45:1cm) [label=below: $\epsilon_1$]{$1$};
            \end{scope}
            \begin{scope}[xshift=\i*2.5 cm,yshift={-4.5 cm}]
            \node [fill=white,inner sep=6pt,shape=circle,draw] at (45:1cm) [label=below: $r$]{$\text{\Qmark}$};
            \end{scope}
        }

        \begin{scope}[xshift=0 cm, yshift={5.1 cm}]
            \node[] at (45:1cm) {$\Tilde{j}$};
        \end{scope}

        \begin{scope}[xshift=3*2.5 cm,yshift={1.5 cm}]
                \draw(0:0) ++ (45:1) ++ (0:-.5cm)  -- ++(0:.45cm);
                \draw(0:0) ++ (45:1) -- ++ (0:0.95cm);
        \end{scope}
            \begin{scope}[xshift=-2*2.5 cm,yshift={1.5 cm}]
                    \draw(0:0) ++ (45:1) -- ++ (0:-0.95cm);
            \end{scope}
        \foreach \i in {-2}{
				\begin{scope}[xshift=\i*2.5 cm,yshift={1.5 cm}]
                    \draw(0:0) ++ (45:1) -- ++ (0:1.95cm);
                    \node [fill=white,inner sep=6pt,shape=circle,draw] at (45:1cm) {$*$};
				\end{scope}
				}
        \foreach \i in {-1}{
				\begin{scope}[xshift=\i*2.5 cm,yshift={1.5 cm}]
                    \draw(0:0) ++ (45:1) -- ++ (0:1.95cm);
                \node [fill=white,inner sep=6pt,shape=circle,draw] at (45:1cm) {$1$};
				\end{scope}
				}
        \foreach \i in {0}{
				\begin{scope}[xshift=\i*2.5 cm,yshift={1.5 cm}]
                    \draw(0:0) ++ (45:1) -- ++ (0:1.95cm);
                    \node [fill=white,inner sep=6pt,shape=circle,draw] at (45:1cm) {$1$};
				\end{scope}
				}
        \foreach \i in {1}{
				\begin{scope}[xshift=\i*2.5 cm,yshift={1.5 cm}]
                    \draw(0:0) ++ (45:1) -- ++ (0:1.95cm);
                    \node [fill=white,inner sep=6pt,shape=circle,draw] at (45:1cm) {$\text{\Qmark}$};
				\end{scope}
				}
            \foreach \i in {2,3}{
            \begin{scope}[xshift=\i*2.5 cm,yshift={1.5 cm}]
            \node [fill=white,inner sep=6pt,shape=circle,draw] at (45:1cm) {$\text{\Qmark}$};
            \end{scope}
            }

        \begin{scope}[xshift=3*2.5 cm,yshift={2.8 cm}]
                \draw(0:0) ++ (45:1) ++ (0:-.5cm)  -- ++(0:.45cm);
                \draw(0:0) ++ (45:1) -- ++ (0:0.95cm);
        \end{scope}

        \begin{scope}[xshift=3 *2.5 cm,yshift={2.8 cm}]
          \draw (0:0) ++ (45:1) -- ++ (0:-1.95cm);
        \end{scope}

        \begin{scope}[xshift=-2*2.5 cm,yshift={2.8 cm}]
                    \draw(0:0) ++ (45:1) -- ++ (0:-0.95cm);
            \end{scope}
        \foreach \i in {-2}{
				\begin{scope}[xshift=\i*2.5 cm,yshift={2.8 cm}]
                    \draw(0:0) ++ (45:1) -- ++ (0:2cm);
                    \node [fill=white,inner sep=6pt,shape=circle,draw] at (45:1cm) {$*$};
				\end{scope}
				}
        \foreach \i in {-1}{
				\begin{scope}[xshift=\i*2.5 cm,yshift={2.8 cm}]
                    \draw(0:0) ++ (45:1) -- ++ (0:1.95cm);
                \node [fill=white,inner sep=6pt,shape=circle,draw] at (45:1cm) {$0$};
				\end{scope}
				}
        \foreach \i in {0}{
				\begin{scope}[xshift=\i*2.5 cm,yshift={2.8cm}]
                    \draw(0:0) ++ (45:1) -- ++ (0:1.95cm);
                    \node [fill=white,inner sep=6pt,shape=circle,draw] at (45:1cm) {$1$};
				\end{scope}
				}
        \foreach \i in {1}{
				\begin{scope}[xshift=\i*2.5 cm,yshift={2.8 cm}]
                    \draw(0:0) ++ (45:1) -- ++ (0:1.95cm);
                    \node [fill=white,inner sep=6pt,shape=circle,draw] at (45:1cm) {$\text{\Qmark}$};
				\end{scope}
				}
            \foreach \i in {2,3}{
              \begin{scope}[xshift=\i*2.5 cm,yshift={2.8 cm}]
            \node [fill=white,inner sep=6pt,shape=circle,draw] at (45:1cm) {$\text{\Qmark}$};
            \end{scope}
            }

        \begin{scope}[xshift=3*2.5 cm,yshift={4.1 cm}]
                \draw(0:0) ++ (45:1) ++ (0:-.5cm)  -- ++(0:.45cm);
                \draw(0:0) ++ (45:1) -- ++ (0:0.95cm);
        \end{scope}

        \begin{scope}[xshift=3 *2.5 cm,yshift={4.1 cm}]
          \draw (0:0) ++ (45:1) -- ++ (0:-1.95cm);
        \end{scope}

        \begin{scope}[xshift=-2*2.5 cm,yshift={4.1 cm}]
                    \draw(0:0) ++ (45:1) -- ++ (0:-0.95cm);
            \end{scope}
        \foreach \i in {-2}{
				\begin{scope}[xshift=\i*2.5 cm,yshift={4.1 cm}]
                    \draw(0:0) ++ (45:1) -- ++ (0:2cm);
                    \node [fill=white,inner sep=6pt,shape=circle,draw] at (45:1cm) {$*$};
				\end{scope}
				}
        \foreach \i in {-1}{
				\begin{scope}[xshift=\i*2.5 cm,yshift={4.1 cm}]
                    \draw(0:0) ++ (45:1) -- ++ (0:1.95cm);
                \node [fill=white,inner sep=6pt,shape=circle,draw] at (45:1cm) {$*$};
				\end{scope}
				}
        \foreach \i in {0}{
				\begin{scope}[xshift=\i*2.5 cm,yshift={4.1 cm}]
                    \draw(0:0) ++ (45:1) -- ++ (0:1.95cm);
                    \node [fill=white,inner sep=6pt,shape=circle,draw] at (45:1cm) {$1$};
				\end{scope}
				}
        \foreach \i in {1}{
				\begin{scope}[xshift=\i*2.5 cm,yshift={4.1 cm}]
                    \draw(0:0) ++ (45:1) -- ++ (0:1.95cm);
                    \node [fill=white,inner sep=6pt,shape=circle,draw] at (45:1cm) {$0$};
				\end{scope}
				}
            \foreach \i in {2,3}{
              \begin{scope}[xshift=\i*2.5 cm,yshift={4.1 cm}]
            \node [fill=white,inner sep=6pt,shape=circle,draw] at (45:1cm) {$\text{\Qmark}$};
            \end{scope}
            }

\end{tikzpicture}
  }
    \caption{Transition probabilities of the different cells when the state $f$ of the boundary belongs to $\{(0,1),(1,0),(1,1)\}$. The first three lines correspond to the three cases $f=(1,0)$, $f=(0,1)$ and $f=(1,1)$. The boundary position $j$ depends on the case, but the modified boundary $\tilde{j}$ is always the same. The boxes provide the different possible states of the cells with their respective probability.}\label{fig:F1}
\end{figure}
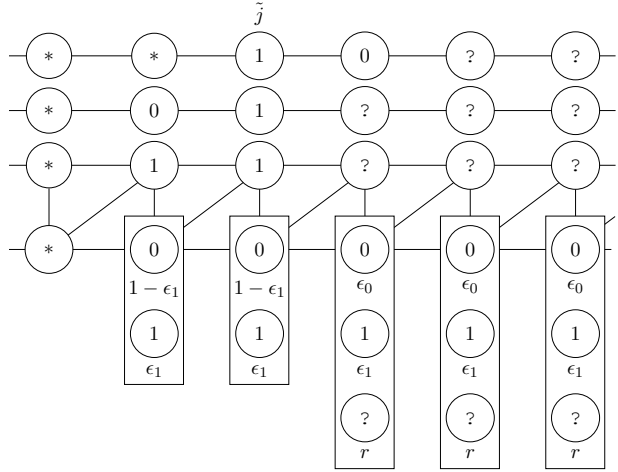

Using Equation~\eqref{eq:trans1}, we can compute the mean increment in one time step for this type of boundary:

\begin{align}
  & \esp{\Tilde{J}_t| f_t = \{(0,1),(1,0),(1,1)\}} \nonumber \\
  & = - \epsilon_1 (1-\epsilon_1)r - \frac{1}{2} (1-\epsilon_1)^2 r + \frac{1}{2} \epsilon_0(1-\epsilon_1)r + \dots + \sum_{k=0}^\infty (k+2) \epsilon_1^2 (1-r)^k r \nonumber \\ 
  & = -\frac{1}{2}r + \epsilon_1 r + \frac{1}{2}\epsilon_0 (1-\epsilon_1)r  + \frac{1}{2} \epsilon_1^2 r  + \frac{1}{2}\epsilon_0^2 +\epsilon_0\epsilon_1 + \epsilon_1^2 + \frac{(1-r)^2}{r} \nonumber \\
  & = -\frac{1}{2} + \underbrace{\frac{1}{2} \epsilon_0 + \frac{1}{2} \epsilon_1 + \epsilon_1 r + \frac{1}{2}\epsilon_0 (1-\epsilon_1)r  + \frac{1}{2} \epsilon_1^2 r  + \frac{1}{2} (1-r)^2 + \frac{1}{2} \epsilon_1^2 + \frac{(1-r)^2}{r}}_{>0}. \label{eq:1stepn21}
\end{align}

This can be checked by hand or with the help of a computer algebra system.

\subsubsection{Case $f_t = (0,0)$:}
For this type of boundary, we look at the increment in two time steps. The first step is given on Figure~\ref{fig:F2}. 

\paragraph{First step:}
As illustrated on Figure~\ref{fig:F2}, if $\Tilde{j}_{t} = \Tilde{j}$ and if the state of the boundary satisfies $f_t = (0,0)$, then
\begin{equation}
(\Tilde{j}_{t+1},f_{t+1}) =
  \begin{cases}
    \left(\geq \Tilde{j}-\frac{3}{2},(*,0)\right) & \text{w.p.\ } \epsilon_0r^2 \\
    (\Tilde{j}-\frac{3}{2} , (1,0)) & \text{w.p.\ } \epsilon_1 \epsilon_0 r \\
    (\Tilde{j}-1, (0,0)) & \text{w.p.\ } \epsilon_0^2 r \\
    (\Tilde{j}-\frac{1}{2}, (0,1)) & \text{w.p.\ } \epsilon_0(1-\epsilon_0)r \\
    (\Tilde{j}-\frac{1}{2} , (1,1)) & \text{w.p.\ } \epsilon_1(1-\epsilon_0)r \\
    (\Tilde{j}-\frac{1}{2} , (*,1)) & \text{w.p.\ } r(1-\epsilon_0)r \\
    (\Tilde{j}-\frac{1}{2} ,(1,0)) & \text{w.p.\ } (1-\epsilon_0) \epsilon_0 r \\
    (\Tilde{j},(0,0)) & \text{w.p.\ }\epsilon_0^2r \\
    (\Tilde{j}+\frac{1}{2},(0,1)) & \text{w.p.\ } \epsilon_0 \epsilon_1 r \\
    (\Tilde{j}+\frac{1}{2},(1,1)) & \text{w.p.\ } (1-\epsilon_0) \epsilon_1 r \\
    (\Tilde{j}+k+\frac{1}{2},(1,0)) & \text{w.p.\ } (\epsilon_0+\epsilon_1)^k \epsilon_1 \epsilon_0 r \\
    (\Tilde{j}+k+1 , (0,0)) & \text{w.p.\ } (\epsilon_0+\epsilon_1)^k \epsilon_0^2 r \\
    (\Tilde{j}+k+\frac{3}{2} , (0,1)) & \text{w.p.\ } (\epsilon_0+\epsilon_1)^k \epsilon_0 \epsilon_1 r \\
    (\Tilde{j}+k+\frac{3}{2}, (1,1)) & \text{w.p.\ } (\epsilon_0+\epsilon_1)^k \epsilon_1^2 r \\
  \end{cases} \label{eq:trans00}
\end{equation}

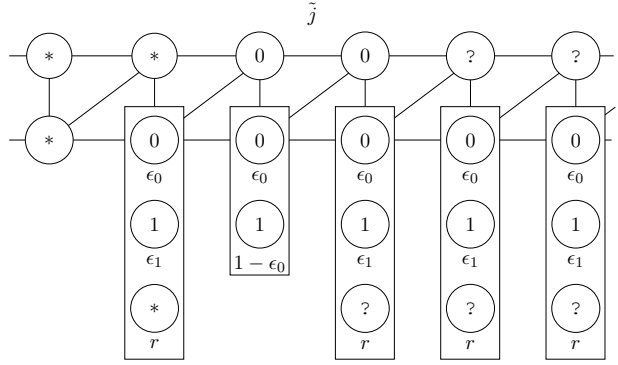
\begin{figure}
    \centering
    \scalebox{0.7}{%
    \begin{tikzpicture}[scale= 0.8]

        \begin{scope}[xshift=-2*2.5 cm,yshift={-0.5 cm}]
            \draw(0:0) ++ (45:1) -- ++ (0:-0.95cm);
        \end{scope}
        
        \begin{scope}[xshift=2.5 cm,yshift={-0.5 cm}]
            \draw (0:0) ++ (45:1) -- ++ (0:1.95cm);
            \draw (0:0) ++ (45:1) -- ++ (90:1.45cm);
            \draw (0:0) ++ (45:1) -- ++ (37:3.05cm);
        \end{scope}
        \begin{scope}[xshift=5 cm,yshift={-0.5 cm}]
            \draw (0:0) ++ (45:1) -- ++ (0:1.95cm);
            \draw (0:0) ++ (45:1) -- ++ (90:1.45cm);
            \draw (0:0) ++ (45:1) -- ++ (37:3.05cm);
        \end{scope}
        \begin{scope}[xshift=7.4 cm,yshift={-0.5 cm}]
            \draw(0:0) ++ (45:1) -- ++ (0:0.95cm);
            \draw (0:0) ++ (45:1) -- ++ (37:1.3cm);
        \end{scope}

        \foreach \i in {-2}{
				\begin{scope}[xshift=\i*2.5 cm,yshift={-0.5 cm}]
                    \draw (0:0) ++ (45:1) -- ++ (0:1.95cm);
                    \draw (0:0) ++ (45:1) -- ++ (90:1.45cm);
                    \draw (0:0) ++ (45:1) -- ++ (37:3.05cm);
                    \node [fill=white,inner sep=5.5pt,shape=circle,draw] at (45:1cm) {};
                    \node [fill=white,inner sep=6.5pt,shape=circle,draw] at (45:1cm) {$*$};
				\end{scope}
        }
                
        \foreach \i in {-1}{
				\begin{scope}[xshift=\i*2.5 cm,yshift={-0.5 cm}]
                    \draw (0:0) ++ (45:1) -- ++ (0:1.95cm);
                    \draw (0:0) ++ (45:1) -- ++ (90:1.45cm);
                    \draw (0:0) ++ (45:1) -- ++ (37:3.05cm);
                    \node [fill=white,inner sep=6pt,shape=circle,draw] at (45:1cm) [label=below: $1-\epsilon_1$] {$0$};
                \end{scope}
                \begin{scope}[xshift=\i*2.5 cm,yshift={-0.5 cm}]
                    \draw (0,-4.5) rectangle +(1.4,6) [fill=white] ;
                    \node [fill=white,inner sep=6pt,shape=circle,draw] at (45:1cm) [label=below: $\epsilon_0$]{$0$};
                \end{scope}
                \begin{scope}[xshift=\i*2.5 cm,yshift={-2.5 cm}]
                    \node [fill=white,inner sep=6pt,shape=circle,draw] at (45:1cm) [label=below:$\epsilon_1$]{$1$};
                \end{scope}
                \begin{scope}[xshift=\i*2.5 cm,yshift={-4.5 cm}]
                    \node [fill=white,inner sep=6.5pt,shape=circle,draw] at (45:1cm) [label=below: $r$]{$*$};
                \end{scope}
                }
        \foreach \i in {0}{
				\begin{scope}[xshift=\i*2.5 cm,yshift={-0.5 cm}]
                    \draw (0:0) ++ (45:1) -- ++ (0:1.95cm);
                    \draw (0:0) ++ (45:1) -- ++ (90:1.45cm);
                    \draw (0:0) ++ (45:1) -- ++ (37:3.05cm);
                    \draw (0,-2.5) rectangle +(1.4,4) [fill=white] ;
                    \node [fill=white,inner sep=5.5pt,shape=circle,draw] at (45:1cm){};
                    \node [fill=white,inner sep=6pt,shape=circle,draw] at (45:1cm) [label=below: $\epsilon_0$] {$0$};
				\end{scope}
                \begin{scope}[xshift=\i*2.5 cm,yshift={-2.5 cm}]
                    \node [fill=white,inner sep=6pt,shape=circle,draw] at (45:1cm) [label=below: $1-\epsilon_0$] {$1$};
                \end{scope}
				}
        \foreach \i in {3}{
            \begin{scope}[xshift=\i*2.5 cm,yshift={-0.5 cm}]
                \draw (0:0) ++ (45:1) -- ++ (90:1.45cm);
            \end{scope}
            \begin{scope}[xshift=\i*2.5 cm,yshift={1.5 cm}]
                \draw (0:0) ++ (45:1) -- ++ (0:-1.95cm);
            \end{scope}
        }
        \foreach \i in {1,...,3}{
            
            \begin{scope}[xshift=\i*2.5 cm,yshift={-0.5 cm}]
            \draw (0,-4.5) rectangle +(1.4,6) [fill=white] ;
            \node [fill=white,inner sep=6pt,shape=circle,draw] at (45:1cm) [label=below: $\epsilon_0$]{$0$};
            \end{scope}
            \begin{scope}[xshift=\i*2.5 cm,yshift={-2.5 cm}]
            \node [fill=white,inner sep=6pt,shape=circle,draw] at (45:1cm) [label=below: $\epsilon_1$]{$1$};
            \end{scope}
            \begin{scope}[xshift=\i*2.5 cm,yshift={-4.5 cm}]
            \node [fill=white,inner sep=6pt,shape=circle,draw] at (45:1cm) [label=below: $r$]{$\text{\Qmark}$};
            \end{scope}
        }

        \begin{scope}[xshift=3*2.5 cm,yshift={1.5 cm}]
                \draw(0:0) ++ (45:1) ++ (0:-.5cm)  -- ++(0:.45cm);
                \draw(0:0) ++ (45:1) -- ++ (0:0.9cm);
        \end{scope}
        
        \begin{scope}[xshift=2*2.5 cm,yshift={1.5 cm}]
        \end{scope}
            \begin{scope}[xshift=-2*2.5 cm,yshift={1.5 cm}]
                    \draw(0:0) ++ (45:1) -- ++ (0:-0.95cm);
            \end{scope}
        \begin{scope}[xshift=1.25 cm, yshift={2.5 cm}]
            \node[] at (45:1cm) {$\Tilde{j}$};
        \end{scope}
        \foreach \i in {-2,-1}{
				\begin{scope}[xshift=\i*2.5 cm,yshift={1.5 cm}]
                    \draw(0:0) ++ (45:1) -- ++ (0:2cm);
                    \node [fill=white,inner sep=6pt,shape=circle,draw] at (45:1cm) {$*$};
				\end{scope}
				}
        \foreach \i in {0,1}{
				\begin{scope}[xshift=\i*2.5 cm,yshift={1.5 cm}]
                    \draw(0:0) ++ (45:1) -- ++ (0:1.95cm);
                    \node [fill=white,inner sep=6pt,shape=circle,draw] at (45:1cm) {$0$};
				\end{scope}
				}
        
            \foreach \i in {2,...,3}{
            \begin{scope}[xshift=\i*2.5 cm,yshift={1.5 cm}]
            \node [fill=white,inner sep=6pt,shape=circle,draw] at (45:1cm) {$\text{\Qmark}$};
            \end{scope}
            }
        
\end{tikzpicture}
    }
    \caption{Transition probabilities when the boundary is in state $(0,0)$ and $\Tilde{j}$ indicates the position of the boundary with the modification considered.}
    \label{fig:F2}
\end{figure}

\begin{remark}
In the first case, which corresponds to a transition to $f_{t+1} = (*,0)$, the value of $\Tilde{j}_{t+1}$ can be equal to $\Tilde{j}-3/2$ if the boundary is in state $(1,0)$ or to $\Tilde{j}-1$ if it is in state $(0,0)$. As we want to bound by below the increment, we keep the lower bound $\Tilde{j} -3/2$ in that case.
\end{remark}

\paragraph{Second step:}
As the second step depends only on the first step through the value of $f_{t+1}$, the transitions given in Equation~\eqref{eq:trans00} allow to deduce the mean increment for two time steps,
\begin{align}
& \esp{\Tilde{J}_{t}+ \Tilde{J}_{t+1}|f_t=(0,0)} \nonumber \\
& =\esp{\Tilde{J}_t|f_t=(0,0)} + \prob{f_{t+1}= (*,0) | f_t=(0,0)} \esp{\Tilde{J}_{t+1}|f_{t+1}=  (*,0)} \nonumber \\
& + \prob{f_{t+1}=(0,0)|f_t=(0,0)} \esp{\Tilde{J}_{t+1}|f_{t+1}=(0,0)} \nonumber \\
& + \prob{f_{t+1}\in \{(0,1),(1,0),(1,1),(*,1)\} |f_t=(0,0)} \esp{\Tilde{J}_{t+1}|f_{t+1} = (0,1)}. \label{eq:2pas}
\end{align}
The probability transitions of $f_{t+1}$ are given by Equation~\eqref{eq:trans00}:
\begin{align*}
& \prob{f_{t+1} = (*,0) |f_t=(0,0)} = \epsilon_0r^2,\\
& \prob{f_{t+1} = (0,0) |f_t=(0,0)} = 2\epsilon_0^2r+ \sum_{k=0}^\infty \epsilon_0^2 (1-r)^k r = 2\epsilon_0^2r + \epsilon_0^2, \text{ and, so, } \\
& \prob{f_{t+1}\in \{(0,1),(1,0),(1,1),(*,1)\} |f_t=(0,0)} = 1 - \epsilon_0 r^2 - 2 \epsilon_0^2 r - \epsilon_0^2.
\end{align*}
Moreover, by Equation~(\ref{eq:trans00}), the mean increment in one time step when the boundary is in state $(0,0)$ verify
\begin{align}
\esp{\Tilde{J}_{t} | f_{t} = (0,0) } \geq -\frac{1}{2} r^2 - 2 \epsilon_0 r+ \epsilon_0^2 r + \frac{1}{2} \epsilon_1^2 + \frac{(1-r)^2}{r} \label{eq:inf00}
\end{align}
and, when it is in state $(*,0)$ we have the following lower bound:
\begin{align}
\esp{\Tilde{J}_t | f_t = (*,0) } &\geq \min \left(\esp{\Tilde{J}_t|f_t=(0,0)} ,\esp{\Tilde{J}_t|f_t = (1,0)}\right) \nonumber\\
&\geq -\frac{1}{2} r^2 - 2 \epsilon_0 r+ \epsilon_0^2 r + \frac{1}{2} \epsilon_1^2 + \frac{(1-r)^2}{r}. \label{eq:inf*0}
\end{align}
Indeed,
\begin{align}
& \esp{\Tilde{J}_t|f_t = (1,0)} - \left(-\frac{1}{2} r^2 - 2 \epsilon_0 r + \epsilon_0^2 r + \frac{1}{2} \epsilon_1^2 + \frac{(1-r)^2}{r} \right) \nonumber \\
& = \frac{1}{2} (1-r)^2 + \frac{1}{2} \epsilon_1 r + \frac{1}{2} \epsilon_1^2 r + \frac{1}{2} \epsilon_0 (1-\epsilon_1)r + \epsilon_0 \left(\frac{3}{2}-\epsilon_0 \right)r \geq 0. \label{eq:diff10}
\end{align}
Hence, as detailed in Appendix~\ref{sec:proofn200}, we finally obtain that
\begin{equation} \label{eq:2stepn200}
\esp{\Tilde{J}_{t}+\Tilde{J}_{t+1} | f_t=(0,0)} > -1.
\end{equation}
Equations~\eqref{eq:1stepn21} and~\eqref{eq:2stepn200} imply Equation~\eqref{eq:goal} when $n=2$, and so Theorem~\ref{thm:main} for $n=2$.

\section{Proof of Theorem~\ref{thm:main} when $n=3$}\label{sec:proof3}
In this Section, we consider the PCA $\Tilde{H}_3$. The proof follows the same way as for $\Tilde{H}_2$. The main difference is that we now consider a boundary $f_t$ of size~$4$. Also, the definition of the modified boundary $\tilde{j}_t$ is slightly different.

\subsection{Modified boundary}\label{sec:mod3}

Let us consider a decorrelated island $(X_{i_t},\dots,X_{j_t})$. We now define the modified boundary position $\Tilde{j}_t$ by
\begin{displaymath}
\Tilde{j}_{t} = \begin{cases}
    j_t & \text{if } X_{j_t}=1, \\
    j_t-1 & \text{if } (X_{j_t-2},X_{j_t-1},X_{j_t})=(0,0,0) \text{ or } (X_{j_t-1},X_{j_t}) = (1,0),\\ 
    j_t-2 & \text{else, i.e.\ if } (X_{j_t-2},X_{j_t-1},X_{j_t})=(1,0,0).
\end{cases}
\end{displaymath}

\subsection{Transition probabilities}

For $n=3$, we consider a right boundary made of the states of the 4 rightmost cells of the island, $f_{t} = (X_{-3+j_t},X_{-2+j_t},X_{-1+j_t},X_{j_t})$, and 
in the following, we assume that $j_t-i_t \geq 9$. It is not a loss of generality to achieve the result, and it allows to avoid any problem of dependency between  $(\Tilde{j}_{t+1},f_t)$ and its left counterpart.

We group the possibles states of the boundary into three subsets of configurations (two being reduced to a singleton): $\{(0,0,0,0)\}$, $\{(1,0,0,0)\}$ and 
\begin{displaymath}
S_1 = \{(a_0,a_1,a_2,a_3) : a_0 \in \{0,1\} \text{ and } \exists k\in \{1,2,3\},\ a_k = 1\}.
\end{displaymath}
The $14$ configurations of this last set can be considered together, as their transition probabilities are the same. This is similar to the case $n=2$, when $(0,1)$, $(1,0)$ and $(1,1)$ were grouped.

\subsubsection{Case $f_t \in S_1$:}
For this case, we can reduce the boundary to a length equal to~$3$, as the mean increment in one time step is enough. Hence, in the transition, we do not care about the value of $X_{j_{t+1}-3}$, as illustrated on Figure~\ref{fig:transS1}. If $\Tilde{j}_t = \Tilde{j}$ and $f_t=f \in S_1$, then, we have 
\begin{align*}
    & \left(\Tilde{j}_{t+1},(X_{j_{t+1}-2},X_{j_{t+1}-1},X_{j_{t+1}}) \right) = \\
    & \begin{cases}
    (\Tilde{j}-1+k,(0,0,0)) & \text{w.p.\ } (1-\epsilon_1)^{3-k} \epsilon_0^k r \text{ for } k\in\{ 0,1,2\};\\
    (\Tilde{j}+2+k,(0,0,0)) & \text{w.p.\ } (\epsilon_0+\epsilon_1)^k \epsilon_0^3 r  \text{ for } k\in\NN;\\
    (\Tilde{j}-l+k,\{0,1\}^{2-l}\times\{1\}\times\{0\}^l) & \text{w.p.\ } \epsilon_1(1-\epsilon_1)^{l-k} \epsilon_0^k r \text{ for } l \in \{0,1,2\} \text{ and } 0\leq k\leq l;\\
    (\Tilde{j}-l+k,\{0,1\}^{2-l}\times\{1\}\times\{0\}^l) & \text{w.p.\ } (\epsilon_0+\epsilon_1)^{k-l-1} \epsilon_1 \epsilon_0^l r \text{ for } l \in \{0,1\} \text{ and } l+1\leq k\leq 2;\\
    (\Tilde{j}+3-l+k,\{0,1\}^{2-l}\times\{1\}\times\{0\}^l) & \text{w.p.\ } (\epsilon_0+\epsilon_1)^{k+2-l} \epsilon_1 \epsilon_0^l r  \text{ for } l\in \{0,1,2\} \text{ and } k\in\NN;
    \end{cases}
\end{align*}

\begin{figure}
    \centering
    \scalebox{0.7}{%
    \begin{tikzpicture}[scale= 0.8]
        
        \begin{scope}[xshift=2.5 cm, yshift={5.1 cm}]
            \node[] at (45:1cm) {$\Tilde{j}$};
        \end{scope}
            
        \begin{scope}[xshift=-2*2.5 cm,yshift={-0.5 cm}]
            \draw(0:0) ++ (45:1) -- ++ (0:-0.95cm);
        \end{scope}
         \begin{scope}[xshift=5 cm,yshift={-0.5 cm}]
            \draw (0:0) ++ (45:1) -- ++ (0:1.95cm);
            \draw (0:0) ++ (45:1) -- ++ (90:1.45cm);
            \draw (0:0) ++ (45:1) -- ++ (37:3.05cm);
            \draw (0:0) ++ (45:1) -- ++ (19:5cm);
        \end{scope}
        \begin{scope}[xshift=7.5 cm,yshift={-0.5 cm}]
            \draw (0:0) ++ (45:1) -- ++ (0:1.95cm);
            \draw (0:0) ++ (45:1) -- ++ (90:1.45cm);
            \draw (0:0) ++ (45:1) -- ++ (37:3.05cm);
            \draw (0:0) ++ (45:1) -- ++ (19:3.6cm);
        \end{scope}
        \begin{scope}[xshift=10 cm,yshift={-0.5 cm}]
            \draw(0:0) ++ (45:1) -- ++ (0:0.95cm);
            \draw (0:0) ++ (45:1) -- ++ (37:1.2cm);
            \draw (0:0) ++ (45:1) -- ++ (19:1cm);
        \end{scope}

        \foreach \i in {-2}{
			\begin{scope}[xshift=\i*2.5 cm,yshift={-0.5 cm}]
                \draw (0:0) ++ (45:1) -- ++ (0:1.95cm);
                \draw (0:0) ++ (45:1) -- ++ (90:1.45cm);
                \draw (0:0) ++ (45:1) -- ++ (37:3.05cm);
                \draw (0:0) ++ (45:1) -- ++ (19:5cm);
                \node [fill=white,inner sep=5.5pt,shape=circle,draw] at (45:1cm){};
                \node [fill=white,inner sep=6.5pt,shape=circle,draw] at (45:1cm)  {$*$};
			\end{scope}
		}
    
        \foreach \i in {-1,0,1}{
				\begin{scope}[xshift=\i*2.5 cm,yshift={-0.5 cm}]
                    \draw (0:0) ++ (45:1) -- ++ (0:1.95cm);
                    \draw (0:0) ++ (45:1) -- ++ (90:1.45cm);
                    \draw (0:0) ++ (45:1) -- ++ (37:3.05cm);
                    \draw (0:0) ++ (45:1) -- ++ (19:5cm);
                    \draw (0,-2.5) rectangle +(1.4,4) [fill=white] ;
                    \node [fill=white,inner sep=5.5pt,shape=circle,draw] at (45:1cm){};
                    \node [fill=white,inner sep=6pt,shape=circle,draw] at (45:1cm) [label=below: $1-\epsilon_1$] {$0$};
				\end{scope}
                \begin{scope}[xshift=\i*2.5 cm,yshift={-2.5 cm}]
                    \node [fill=white,inner sep=6pt,shape=circle,draw] at (45:1cm) [label=below: $\epsilon_1$] {$1$};
                \end{scope}
				}

        \foreach \i in {2,...,4}{
            
            \begin{scope}[xshift=\i*2.5 cm,yshift={-0.5 cm}]
            \draw (0,-4.5) rectangle +(1.4,6) [fill=white] ;
            \node [fill=white,inner sep=6pt,shape=circle,draw] at (45:1cm) [label=below: $\epsilon_0$]{$0$};
            \end{scope}
            \begin{scope}[xshift=\i*2.5 cm,yshift={-2.5 cm}]
            \node [fill=white,inner sep=6pt,shape=circle,draw] at (45:1cm) [label=below: $\epsilon_1$]{$1$};
            \end{scope}
            \begin{scope}[xshift=\i*2.5 cm,yshift={-4.5 cm}]
            \node [fill=white,inner sep=6pt,shape=circle,draw] at (45:1cm) [label=below: $r$]{$\text{\Qmark}$};
            \end{scope}
        }

        \begin{scope}[xshift=4*2.5 cm,yshift={1.5 cm}]
                \draw(0:0) ++ (45:1) ++ (0:-.5cm)  -- ++(0:.45cm);
                \draw(0:0) ++ (45:1) -- ++ (0:0.95cm);
        \end{scope}
        \begin{scope}[xshift=3*2.5 cm,yshift={1.5 cm}]
            \draw(0:0) ++ (45:1) -- ++ (0:1.95cm);
        \end{scope}
        \begin{scope}[xshift=-2*2.5 cm,yshift={1.5 cm}]
                \draw(0:0) ++ (45:1) -- ++ (0:-0.95cm);
        \end{scope}
        \foreach \i in {-2,...,0}{
				\begin{scope}[xshift=\i*2.5 cm,yshift={1.5 cm}]
                    \draw(0:0) ++ (45:1) -- ++ (0:1.95cm);
                    \node [fill=white,inner sep=4pt,shape=circle,draw] at (45:1cm) {$0|1$};
				\end{scope}
				}
        \foreach \i in {1}{
				\begin{scope}[xshift=\i*2.5 cm,yshift={1.5 cm}]
                    \draw(0:0) ++ (45:1) -- ++ (0:1.95cm);
                    \node [fill=white,inner sep=6pt,shape=circle,draw] at (45:1cm) {$1$};
				\end{scope}
				}
		\foreach \i in {2}{
          \begin{scope}[xshift=\i*2.5 cm,yshift={1.5 cm}]
            \draw(0:0) ++ (45:1) -- ++ (0:1.95cm);
          \end{scope}
        }	
        \foreach \i in {2,3}{
          \begin{scope}[xshift=\i*2.5 cm,yshift={1.5 cm}]
            \node [fill=white,inner sep=6pt,shape=circle,draw] at (45:1cm) {$\text{\Qmark}$};
          \end{scope}
        }
        \foreach \i in {4}{
          \begin{scope}[xshift=\i*2.5 cm,yshift={1.5 cm}]
            \node [fill=white,inner sep=6pt,shape=circle,draw] at (45:1cm) {$\text{\Qmark}$};
          \end{scope}
        }

        \begin{scope}[xshift=4*2.5 cm,yshift={2.8 cm}]
                \draw(0:0) ++ (45:1) ++ (0:-.5cm)  -- ++(0:.45cm);
                \draw(0:0) ++ (45:1) -- ++ (0:0.95cm);
        \end{scope}
        \begin{scope}[xshift=3*2.5 cm,yshift={2.8 cm}]
            \draw(0:0) ++ (45:1) -- ++ (0:1.95cm);
        \end{scope}
            \begin{scope}[xshift=-2*2.5 cm,yshift={2.8 cm}]
                    \draw(0:0) ++ (45:1) -- ++ (0:-0.95cm);
            \end{scope}
            \foreach \i in {-2}{
              \begin{scope}[xshift=\i*2.5 cm,yshift={2.8 cm}]
                \draw(0:0) ++ (45:1) -- ++ (0:1.95cm);
                \node [fill=white,inner sep=6pt,shape=circle,draw] at (45:1cm) {$*$};
              \end{scope}
            }

            \foreach \i in {-1,...,0}{
				\begin{scope}[xshift=\i*2.5 cm,yshift={2.8 cm}]
                    \draw(0:0) ++ (45:1) -- ++ (0:1.95cm);
                    \node [fill=white,inner sep=4pt,shape=circle,draw] at (45:1cm) {$0|1$};
				\end{scope}
				}
        \foreach \i in {1}{
				\begin{scope}[xshift=\i*2.5 cm,yshift={2.8 cm}]
                    \draw(0:0) ++ (45:1) -- ++ (0:1.95cm);
                    \node [fill=white,inner sep=6pt,shape=circle,draw] at (45:1cm) {$1$};
				\end{scope}
				}
        \foreach \i in {2}{
				\begin{scope}[xshift=\i*2.5 cm,yshift={2.8 cm}]
                    \draw(0:0) ++ (45:1) -- ++ (0:1.95cm);
                    \node [fill=white,inner sep=6pt,shape=circle,draw] at (45:1cm) {$0$};
				\end{scope}
                }
                
        \foreach \i in {3,...,4}{
            \begin{scope}[xshift=\i*2.5 cm,yshift={2.8 cm}]
            \node [fill=white,inner sep=6pt,shape=circle,draw] at (45:1cm) {$\text{\Qmark}$};
            \end{scope}
            }

        \begin{scope}[xshift=4*2.5 cm,yshift={4.1 cm}]
                \draw(0:0) ++ (45:1) ++ (0:-.5cm)  -- ++(0:.45cm);
                \draw(0:0) ++ (45:1) -- ++ (0:0.95cm);
        \end{scope}
        \begin{scope}[xshift=3*2.5 cm,yshift={4.1 cm}]
            \draw(0:0) ++ (45:1) -- ++ (0:1.95cm);
        \end{scope}
            \begin{scope}[xshift=-2*2.5 cm,yshift={4.1 cm}]
                    \draw(0:0) ++ (45:1) -- ++ (0:-0.95cm);
            \end{scope}
            \foreach \i in {-2,-1}{
              \begin{scope}[xshift=\i*2.5 cm,yshift={4.1 cm}]
                \draw(0:0) ++ (45:1) -- ++ (0:2cm);
                \node [fill=white,inner sep=6pt,shape=circle,draw] at (45:1cm) {$*$};
              \end{scope}
            }

            \foreach \i in {0}{
				\begin{scope}[xshift=\i*2.5 cm,yshift={4.1 cm}]
                    \draw(0:0) ++ (45:1) -- ++ (0:1.95cm);
                    \node [fill=white,inner sep=4pt,shape=circle,draw] at (45:1cm) {$0|1$};
				\end{scope}
				}
        \foreach \i in {1}{
				\begin{scope}[xshift=\i*2.5 cm,yshift={4.1 cm}]
                    \draw(0:0) ++ (45:1) -- ++ (0:1.95cm);
                    \node [fill=white,inner sep=6pt,shape=circle,draw] at (45:1cm) {$1$};
				\end{scope}
				}
		\foreach \i in {2}{
				\begin{scope}[xshift=\i*2.5 cm,yshift={4.1 cm}]
                    \draw(0:0) ++ (45:1) -- ++ (0:1.95cm);
				\end{scope}
                }
                
        \foreach \i in {2,3}{
				\begin{scope}[xshift=\i*2.5 cm,yshift={4.1 cm}]
                    \node [fill=white,inner sep=6pt,shape=circle,draw] at (45:1cm) {$0$};
				\end{scope}
                }
                
        \foreach \i in {4}{
            \begin{scope}[xshift=\i*2.5 cm,yshift={4.1 cm}]
            \node [fill=white,inner sep=6pt,shape=circle,draw] at (45:1cm) {$\text{\Qmark}$};
            \end{scope}
            }
            
\end{tikzpicture}
    }
    \caption{Transition probabilities of the different cells when the state $f$ of the boundary belongs to $S_1$. Precisely, the first line corresponds to $f \in \{(0,1,0,0),(1,1,0,0\}$, the second line to $f \in \{0,1\}^2 \times \{1\} \times \{0\}$, and the third line to $f \in \{0,1\}^3 \times \{1\}$. The boundary position $j$ depends on the case, but the modified boundary $\tilde j$ is always the same.}
    \label{fig:transS1}
    
    \bigskip
\end{figure}
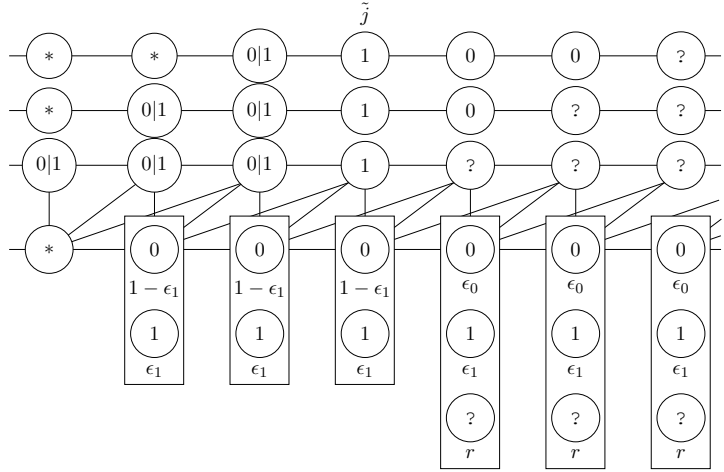

We then obtain the following mean increment:
\begin{align}
& \esp{\Tilde{J_t}|f_t\in S_1} = \sum_{k=0}^2 (-1+k) (1-\epsilon_1)^{3-k} \epsilon_0^k r + \dots \nonumber \\
& =-1+\epsilon_0+2\epsilon_1+\epsilon_0^2+3\epsilon_1^2 ( 1 -\epsilon_1) +\epsilon_0\epsilon_1(1-\epsilon_0)+\epsilon_1^4+\epsilon_0^3\epsilon_1+\frac{(1-r)^{3}}{r} > -1. \label{eq:driftS1} 
\end{align}
This can be checked by hand or with the help of a computer algebra system (see Appendix 7). Hence, we proved Equation~\eqref{eq:goal} for any $f \in S_1$. For the two others cases ($f = (0,0,0,0)$ or $f=(1,0,0,0)$), we have to consider two time steps. The case $f_t = (1,0,0,0)$ is considered first.

\subsubsection{Case $f_t = (1,0,0,0)$:}

As illustrated on Figure~\ref{fig:trans1000}, if $\Tilde{j}_t = \Tilde{j}$ and $f_t = (1,0,0,0)$, then
\begin{align}
    & (\Tilde{j}_{t+1},f_{t+1}) = \nonumber \\ 
    & \begin{cases}
    (\Tilde{j}-2+k,(0,0,0,0)) & \text{w.p.\ }(1-\epsilon_1)^{3-k} \epsilon_0^{1+k} r \text{ for } k\in\{0,1,2,3\};\\
    (\Tilde{j}-2+k,(1,0,0,0)) & \text{w.p.\ } \epsilon_1(1-\epsilon_1)^{2-k} \epsilon_0^{1+k} r \text{ for } k\in\{0,1,2\};\\
    (\Tilde{j}+1,(1,0,0,0)) & \text{w.p.\ } (1-\epsilon_0)\epsilon_0^3r;\\
    (\Tilde{j}-1+k-l,\{0,1\}^{3-l}\times\{1\}\times\{0\}^{l}) & \text{w.p.\ } \epsilon_0^{k+1}\epsilon_1(1-\epsilon_1)^{l-(k+1)}r \text{ for } l \in \{1,2\} \text{ and } 0\leq k\leq l-1;\\
    (\Tilde{j}-1,\{0,1\}^{3-l}\times\{1\}\times\{0\}^{l}) & \text{w.p.\ } (1-\epsilon_0) \epsilon_0^{l}r \text{ for } l \in \{0,1,2\};\\
    (\Tilde{j}-1+k-l,\{0,1\}^{3-l}\times\{1\}\times\{0\}^{l}) & \text{w.p.\ } (\epsilon_0+\epsilon_1)^{k-(l+1)} \epsilon_1 \epsilon_0^l r \text{ for } l \in \{0,1,2\} \text{ and } l+1\leq k \leq 3;\\        
    (\Tilde{j}+2+k,(0,0,0,0)) & \text{w.p.\ } (\epsilon_0+\epsilon_1)^k \epsilon_0^4 r  \text{ for } k\in\NN;\\
    (\Tilde{j}+2+k,(1,0,0,0)) & \text{w.p.\ } (\epsilon_0+\epsilon_1)^k \epsilon_1 \epsilon_0^3 r  \text{ for } k\in\NN;\\
    (\Tilde{j}+3+k-l,\{0,1\}^{3-l}\times\{1\}\times\{0\}^{l}) & \text{w.p.\ } (\epsilon_0+\epsilon_1)^{3+k-l}\epsilon_1 \epsilon_0^l r  \text{ for } l \in \{0,1,2\} \text{ and } k\in\NN;
    \end{cases} \label{eq:trans1000}
\end{align}

We obtain then the following formula for the mean increment in two time steps when $f_t = (1,0,0,0)$:
\begin{align}
    & \esp{\Tilde{J}_{t}+\Tilde{J}_{t+1}|f_t=(1,0,0,0)} \nonumber \\
    & = \esp{\Tilde{J}_{t}|f_t=(1,0,0,0)} + \prob{f_{t+1} \in S_1 |f_t=(1,0,0,0)} \esp{\Tilde{J}_{t+1}|f_{t+1}\in S_1} \nonumber \\
    & + \prob{f_{t+1}=(1,0,0,0) |f_t=(1,0,0,0)} \esp{\Tilde{J}_{t+1}|f_{t+1}=(1,0,0,0)} \nonumber \\
    & + \prob{f_{t+1}=(0,0,0,0) |f_t=(1,0,0,0)} \esp{\Tilde{J}_{t+1}|f_{t+1}=(0,0,0,0)} \label{eq:driftAux1000}
\end{align}

From Equations~\eqref{eq:driftS1} and~\eqref{eq:trans1000}, we have computed or we can compute all the previous terms except $\esp{\Tilde{J}_{t+1}|f_{t+1}=(0,0,0,0)}$. In particular, with the help of a computer algebra system (see Appendix 7), one finds 
\begin{equation}
\esp{\Tilde{J}_{t}|f_{t}=(1,0,0,0)} = -1-\epsilon_0+\epsilon_1+\epsilon_0^2+\epsilon_1^2+2\epsilon_0\epsilon_1+2\epsilon_0^3+2\epsilon_0\epsilon_1^2-\epsilon_0^4-\epsilon_0\epsilon_1^3+\frac{(1-r)^{3}}{r}. \label{eq:1step1000}
\end{equation}

\begin{figure}
    \centering

    \scalebox{0.7}{%
    \begin{tikzpicture}[scale= 0.8]
        
        \begin{scope}[xshift=2.5 cm, yshift={2.5 cm}]
            \node[] at (45:1cm) {$\Tilde{j}$};
        \end{scope}
        \begin{scope}[xshift=-3*2.5 cm,yshift={-0.5 cm}]
            \draw(0:0) ++ (45:1) -- ++ (0:-0.95cm);
        \end{scope}
        \begin{scope}[xshift=2.5 cm,yshift={-0.5 cm}]
            \draw (0:0) ++ (45:1) -- ++ (0:1.95cm);
            \draw (0:0) ++ (45:1) -- ++ (90:1.45cm);
            \draw (0:0) ++ (45:1) -- ++ (37:3.05cm);
            \draw (0:0) ++ (45:1) -- ++ (19:5cm);
        \end{scope}
        \begin{scope}[xshift=5 cm,yshift={-0.5 cm}]
            \draw (0:0) ++ (45:1) -- ++ (0:1.95cm);
            \draw (0:0) ++ (45:1) -- ++ (90:1.45cm);
            \draw (0:0) ++ (45:1) -- ++ (37:3.05cm);
            \draw (0:0) ++ (45:1) -- ++ (19:5cm);
        \end{scope}
        \begin{scope}[xshift=7.5 cm,yshift={-0.5 cm}]
            \draw (0:0) ++ (45:1) -- ++ (0:1.95cm);%
            \draw (0:0) ++ (45:1) -- ++ (90:1.45cm);
            \draw (0:0) ++ (45:1) -- ++ (37:3.05cm);
            \draw (0:0) ++ (45:1) -- ++ (19:3.6cm);
        \end{scope}
        \begin{scope}[xshift=10 cm,yshift={-0.5 cm}]
            \draw(0:0) ++ (45:1) -- ++ (0:0.95cm);
            \draw (0:0) ++ (45:1) -- ++ (37:1.3cm);
            \draw (0:0) ++ (45:1) -- ++ (19:1cm);
        \end{scope}

        \foreach \i in {-3,-2,-1}{
				\begin{scope}[xshift=\i*2.5 cm,yshift={-0.5 cm}]
                    \draw (0:0) ++ (45:1) -- ++ (0:1.95cm);
                    \draw (0:0) ++ (45:1) -- ++ (90:1.45cm);
                    \draw (0:0) ++ (45:1) -- ++ (37:3.05cm);
                    \draw (0:0) ++ (45:1) -- ++ (19:5cm);
                    \draw (0,-2.5) rectangle +(1.4,4) [fill=white] ;
                    \node [fill=white,inner sep=6pt,shape=circle,draw] at (45:1cm) [label=below: $1-\epsilon_1$] {$0$};
                \end{scope}
                \begin{scope}[xshift=\i*2.5 cm,yshift={-2.5 cm}]
                    \node [fill=white,inner sep=6pt,shape=circle,draw] at (45:1cm) [label=below: $\epsilon_1$] {$1$};
				\end{scope}
                }
        \foreach \i in {0}{
				\begin{scope}[xshift=\i*2.5 cm,yshift={-0.5 cm}]
                    \draw (0:0) ++ (45:1) -- ++ (0:1.95cm);
                    \draw (0:0) ++ (45:1) -- ++ (90:1.45cm);
                    \draw (0:0) ++ (45:1) -- ++ (37:3.05cm);
                    \draw (0:0) ++ (45:1) -- ++ (19:5cm);
                    \draw (0,-2.5) rectangle +(1.4,4) [fill=white] ;
                    \node [fill=white,inner sep=5.5pt,shape=circle,draw] at (45:1cm){};
                    \node [fill=white,inner sep=6pt,shape=circle,draw] at (45:1cm) [label=below: $\epsilon_0$] {$0$};
				\end{scope}
                \begin{scope}[xshift=\i*2.5 cm,yshift={-2.5 cm}]
                    \node [fill=white,inner sep=6pt,shape=circle,draw] at (45:1cm) [label=below: $1-\epsilon_0$] {$1$};
                \end{scope}
				}
        \foreach \i in {4}{
            \begin{scope}[xshift=\i*2.5 cm,yshift={-0.5 cm}]
            \draw (0:0) ++ (45:1) -- ++ (90:1.45cm);
            \end{scope}
        }
        \foreach \i in {1,...,4}{
            
            \begin{scope}[xshift=\i*2.5 cm,yshift={-0.5 cm}]
            \draw (0,-4.5) rectangle +(1.4,6) [fill=white] ;
            \node [fill=white,inner sep=6pt,shape=circle,draw] at (45:1cm) [label=below: $\epsilon_0$]{$0$};
            \end{scope}
            \begin{scope}[xshift=\i*2.5 cm,yshift={-2.5 cm}]
            \node [fill=white,inner sep=6pt,shape=circle,draw] at (45:1cm) [label=below: $\epsilon_1$]{$1$};
            \end{scope}
            \begin{scope}[xshift=\i*2.5 cm,yshift={-4.5 cm}]
            \node [fill=white,inner sep=6pt,shape=circle,draw] at (45:1cm) [label=below: $r$]{$\text{\Qmark}$};
            \end{scope}
        }

        \begin{scope}[xshift=4*2.5 cm,yshift={1.5 cm}]
                \draw(0:0) ++ (45:1) ++ (0:-.5cm)  -- ++(0:.45cm);
                \draw(0:0) ++ (45:1) -- ++ (0:0.95cm);
        \end{scope}
        \begin{scope}[xshift=3*2.5 cm,yshift={1.5 cm}]
            \draw(0:0) ++ (45:1) -- ++ (0:1.95cm);
        \end{scope}
        \begin{scope}[xshift=-3*2.5 cm,yshift={1.5 cm}]
            \draw(0:0) ++ (45:1) -- ++ (0:-0.95cm);
        \end{scope}
        \foreach \i in {-3,...,-2}{
				\begin{scope}[xshift=\i*2.5 cm,yshift={1.5 cm}]
                    \draw(0:0) ++ (45:1) -- ++ (0:2cm);
                    \node [fill=white,inner sep=6pt,shape=circle,draw] at (45:1cm) {$*$};
				\end{scope}
				}
        \foreach \i in {-1}{
				\begin{scope}[xshift=\i*2.5 cm,yshift={1.5 cm}]
                    \draw(0:0) ++ (45:1) -- ++ (0:1.95cm);
                    \node [fill=white,inner sep=6pt,shape=circle,draw] at (45:1cm) {$1$};
				\end{scope}
				}
        \foreach \i in {0,...,2}{
				\begin{scope}[xshift=\i*2.5 cm,yshift={1.5 cm}]
                    \draw(0:0) ++ (45:1) -- ++ (0:1.95cm);
                    \node [fill=white,inner sep=6pt,shape=circle,draw] at (45:1cm) {$0$};
				\end{scope}
				}
        
            \foreach \i in {3,...,4}{
            \begin{scope}[xshift=\i*2.5 cm,yshift={1.5 cm}]
            \node [fill=white,inner sep=6pt,shape=circle,draw] at (45:1cm) {$\text{\Qmark}$};
            \end{scope}
            }
        \end{tikzpicture}}
    \caption{Transition probabilities when the boundary is in state $(1,0,0,0)$ and $\Tilde{j}$ the position of the modified boundary.} \label{fig:trans1000}

\end{figure}
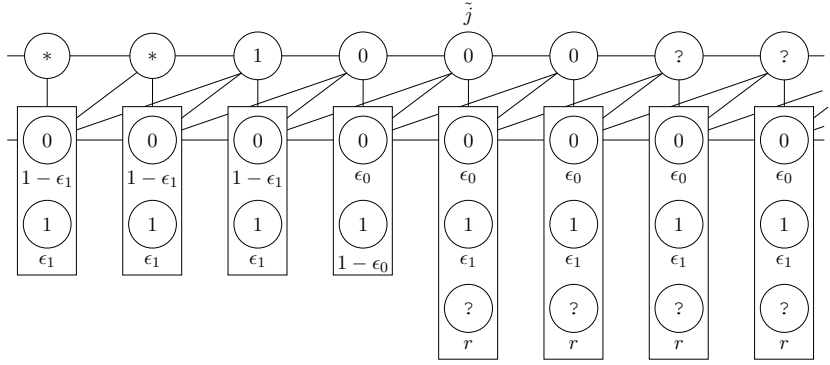

For the missing term, we have to look at the transitions when $f_t = (0,0,0,0)$.

\subsubsection{Case $f_t = (0,0,0,0)$:}

If $\tilde{j}_t = j$ and $f_t = (0,0,0,0)$, then
\begin{align}
    & (\Tilde{j}_{t+1},f_{t+1}) = \nonumber\\
    & \begin{cases}
    (\geq \Tilde{j}-3,(1,*,0,0)) & \text{w.p.\ } \epsilon_1 r \epsilon_0^2 r;\\
    (\geq \Tilde{j}-3,\{0,*\}\times\{*\}\times\{0\}^2) & \text{w.p.\ } (1-\epsilon_1) r \epsilon_0^2 r;\\
    (\Tilde{j}-2+k,(*,0,0,0)) & \text{w.p.\ } r \epsilon_0^3 r \text{ for } k\in\{0,1\};\\
    (\Tilde{j}-3,\{0,1,*\}\times\{1\}\times\{0\}^2) & \text{w.p.\ }     \epsilon_1 \epsilon_0^2 r;\\
    (\Tilde{j}-2,\{0,1,*\}\times\{1\}\times\{0\}^2) & \text{w.p.\ }     (1-\epsilon_0) \epsilon_0^2 r;\\
    (\Tilde{j}-2+k,(0,0,0,0)) & \text{w.p.\ } \epsilon_0^4r \text{ for } k\in\{0,1,2,3\};\\
    (\Tilde{j}-2+k,(1,0,0,0)) & \text{w.p.\ } \epsilon_1 \epsilon_0^3 r \text{ for } k\in\{0,1\};\\
    (\Tilde{j}-2+k,(1,0,0,0)) & \text{w.p.\ }      (1-\epsilon_0) \epsilon_0^3 r \text{ for } k \in\{2,3\};\\
    (\Tilde{j}-1,\{0,1,*\}^2\times\{0,1\}\times\{1\}) & \text{w.p.\ } (1-\epsilon_0)r;\\
    (\Tilde{j},\{0,1,*\}\times\{0,1\}^2\times\{1\}) & \text{w.p.\ } \epsilon_1r;\\
    (\Tilde{j}-2,\{0,1,*\}^2\times\{1\}\times\{0\}) & \text{w.p.\ } (1-\epsilon_0) \epsilon_0 r;\\
    (\Tilde{j}-1,\{0,1,*\}\times\{0,1\}\times\{1\}\times\{0\}) & \text{w.p.\ } (1-\epsilon_0) \epsilon_0 r;\\
    (\Tilde{j}+1+k,\{0,1\}^3\times\{1\}) & \text{w.p.\ } (\epsilon_0+\epsilon_1)^{k+1} \epsilon_1 r \text{ for } k\in\{0,1\};\\
    (\Tilde{j}+k,\{0,1\}^2\times\{1\}\times\{0\}) & \text{w.p.\ } (\epsilon_0+\epsilon_1)^{k} \epsilon_1 \epsilon_0 r \text{ for } k\in\{0,1\};\\
    (\Tilde{j}-1,\{0,1\}\times\{1\}\times\{0\}^2) & \text{w.p.\ } (1-\epsilon_0) \epsilon_0^2 r\\
    (\Tilde{j},\{0,1\}\times\{1\}\times\{0\}^2) & \text{w.p.\ } \epsilon_1 \epsilon_0^2 r;\\
    (\Tilde{j}+2+k,(0,0,0,0)) & \text{w.p.\ } (\epsilon_0+\epsilon_1)^k \epsilon_0^4 r  \text{ for } k\in\NN;\\
    (\Tilde{j}+2+k,(1,0,0,0)) & \text{w.p.\ } (\epsilon_0+\epsilon_1)^k \epsilon_1 \epsilon_0^3 r  \text{ for } k\in\NN;\\
    (\Tilde{j}+3-l+k,\{0,1\}^{4-l}\times\{1\}\times\{0\}^{l}) & \text{w.p.\ } (\epsilon_0+\epsilon_1)^{k+3-l} \epsilon_1 \epsilon_0^{l} r \text{ for } l \in \{0,1,2\} \text{ and } k\in\NN.
    \end{cases} \label{eq:trans0000}
\end{align}

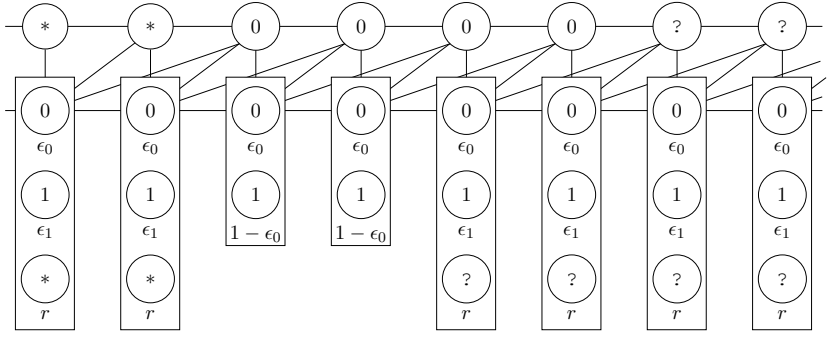
\begin{figure}
    \centering
    \scalebox{0.7}{%
    \begin{tikzpicture}[scale= 0.8]

        \begin{scope}[xshift=-3*2.5 cm,yshift={-0.5 cm}]
            \draw(0:0) ++ (45:1) -- ++ (0:-0.95cm);
        \end{scope}
        \begin{scope}[xshift=2.5 cm,yshift={-0.5 cm}]
            \draw (0:0) ++ (45:1) -- ++ (0:1.95cm);
            \draw (0:0) ++ (45:1) -- ++ (90:1.45cm);
            \draw (0:0) ++ (45:1) -- ++ (37:3.05cm);
            \draw (0:0) ++ (45:1) -- ++ (19:5cm);
        \end{scope}
        \begin{scope}[xshift=5 cm,yshift={-0.5 cm}]
            \draw (0:0) ++ (45:1) -- ++ (0:1.95cm);
            \draw (0:0) ++ (45:1) -- ++ (90:1.45cm);
            \draw (0:0) ++ (45:1) -- ++ (37:3.05cm);
            \draw (0:0) ++ (45:1) -- ++ (19:5cm);
        \end{scope}
        \begin{scope}[xshift=7.5 cm,yshift={-0.5 cm}]
            \draw (0:0) ++ (45:1) -- ++ (0:1.95cm);
            \draw (0:0) ++ (45:1) -- ++ (90:1.45cm);
            \draw (0:0) ++ (45:1) -- ++ (37:3.05cm);
            \draw (0:0) ++ (45:1) -- ++ (19:3.6cm);
        \end{scope}
        \begin{scope}[xshift=10 cm,yshift={-0.5 cm}]
            \draw(0:0) ++ (45:1) -- ++ (0:0.95cm);
            \draw (0:0) ++ (45:1) -- ++ (37:1.3cm);
            \draw (0:0) ++ (45:1) -- ++ (19:1cm);
        \end{scope}

        \foreach \i in {-3,-2}{
				\begin{scope}[xshift=\i*2.5 cm,yshift={-0.5 cm}]
                    \draw (0:0) ++ (45:1) -- ++ (0:1.95cm);
                    \draw (0:0) ++ (45:1) -- ++ (90:1.45cm);
                    \draw (0:0) ++ (45:1) -- ++ (37:3.05cm);
                    \draw (0:0) ++ (45:1) -- ++ (19:5cm);
                    \node [fill=white,inner sep=6pt,shape=circle,draw] at (45:1cm) [label=below: $1-\epsilon_1$] {$0$};
                \end{scope}
                \begin{scope}[xshift=\i*2.5 cm,yshift={-0.5 cm}]
                    \draw (0,-4.5) rectangle +(1.4,6) [fill=white] ;
                    \node [fill=white,inner sep=6pt,shape=circle,draw] at (45:1cm) [label=below: $\epsilon_0$]{$0$};
                \end{scope}
                \begin{scope}[xshift=\i*2.5 cm,yshift={-2.5 cm}]
                    \node [fill=white,inner sep=6pt,shape=circle,draw] at (45:1cm) [label=below:$\epsilon_1$]{$1$};
                \end{scope}
                \begin{scope}[xshift=\i*2.5 cm,yshift={-4.5 cm}]
                    \node [fill=white,inner sep=6.5pt,shape=circle,draw] at (45:1cm) [label=below: $r$]{$*$};
                \end{scope}
                }
        \foreach \i in {-1,0}{
				\begin{scope}[xshift=\i*2.5 cm,yshift={-0.5 cm}]
                    \draw (0:0) ++ (45:1) -- ++ (0:1.95cm);
                    \draw (0:0) ++ (45:1) -- ++ (90:1.45cm);
                    \draw (0:0) ++ (45:1) -- ++ (37:3.05cm);
                    \draw (0:0) ++ (45:1) -- ++ (19:5cm);
                    \draw (0,-2.5) rectangle +(1.4,4) [fill=white] ;
                    \node [fill=white,inner sep=5.5pt,shape=circle,draw] at (45:1cm){};
                    \node [fill=white,inner sep=6pt,shape=circle,draw] at (45:1cm) [label=below: $\epsilon_0$] {$0$};
				\end{scope}
                \begin{scope}[xshift=\i*2.5 cm,yshift={-2.5 cm}]
                    \node [fill=white,inner sep=6pt,shape=circle,draw] at (45:1cm) [label=below: $1-\epsilon_0$] {$1$};
                \end{scope}
				}
        \foreach \i in {4}{
            \begin{scope}[xshift=\i*2.5 cm,yshift={-0.5 cm}]
            \draw (0:0) ++ (45:1) -- ++ (90:1.45cm);
            \end{scope}
        }
        \foreach \i in {1,...,4}{
            
            \begin{scope}[xshift=\i*2.5 cm,yshift={-0.5 cm}]
            \draw (0,-4.5) rectangle +(1.4,6) [fill=white] ;
            \node [fill=white,inner sep=6pt,shape=circle,draw] at (45:1cm) [label=below: $\epsilon_0$]{$0$};
            \end{scope}
            \begin{scope}[xshift=\i*2.5 cm,yshift={-2.5 cm}]
            \node [fill=white,inner sep=6pt,shape=circle,draw] at (45:1cm) [label=below: $\epsilon_1$]{$1$};
            \end{scope}
            \begin{scope}[xshift=\i*2.5 cm,yshift={-4.5 cm}]
            \node [fill=white,inner sep=6pt,shape=circle,draw] at (45:1cm) [label=below: $r$]{$\text{\Qmark}$};
            \end{scope}
        }

        \begin{scope}[xshift=4*2.5 cm,yshift={1.5 cm}]
                \draw(0:0) ++ (45:1) ++ (0:-.5cm)  -- ++(0:.45cm);
                \draw(0:0) ++ (45:1) -- ++ (0:0.95cm);
        \end{scope}
        \begin{scope}[xshift=3*2.5 cm,yshift={1.5 cm}]
            \draw(0:0) ++ (45:1) -- ++ (0:1.95cm);
        \end{scope}
            \begin{scope}[xshift=-3*2.5 cm,yshift={1.5 cm}]
                    \draw(0:0) ++ (45:1) -- ++ (0:-0.95cm);
            \end{scope}
        \foreach \i in {-3,...,-2}{
				\begin{scope}[xshift=\i*2.5 cm,yshift={1.5 cm}]
                    \draw(0:0) ++ (45:1) -- ++ (0:2cm);
                    \node [fill=white,inner sep=6pt,shape=circle,draw] at (45:1cm) {$*$};
				\end{scope}
				}
        \foreach \i in {-1,...,2}{
				\begin{scope}[xshift=\i*2.5 cm,yshift={1.5 cm}]
                    \draw(0:0) ++ (45:1) -- ++ (0:2cm);
                    \node [fill=white,inner sep=6pt,shape=circle,draw] at (45:1cm) {$0$};
				\end{scope}
				}
        
            \foreach \i in {3,...,4}{
            \begin{scope}[xshift=\i*2.5 cm,yshift={1.5 cm}]
            \node [fill=white,inner sep=6pt,shape=circle,draw] at (45:1cm) {$\text{\Qmark}$};
            \end{scope}
            }
        
\end{tikzpicture}}
    \caption{Transition probabilities when the boundary is in state $(0,0,0,0)$ and $\Tilde{j}$ the position of the modified boundary.} \label{fig:trans0000}
\end{figure}

This allows to find, with the help of a computer algebra system (see Appendix 7), the following lower bound, denoted by $I_0$, for $\esp{\Tilde{J}_t | f_t=(0,0,0,0)}$:
\begin{align}
&\esp{\Tilde{J}_t|f_t=(0,0,0,0)} \nonumber \\
&\geq -1-\epsilon_0+\epsilon_1-\epsilon_0^2+\epsilon_1^2+3\epsilon_0\epsilon_1+6\epsilon_0^3 +2\epsilon_0^2\epsilon_1-3\epsilon_0^{4}-3\epsilon_0^3\epsilon_1+\frac{(1-r)^{3}}{r} = I_0. \label{eq:I0}
\end{align}
Unfortunetaly, the lower bound $I_0$ is not greater than $-1$ for any $\epsilon_0$ and $\epsilon_1$. 

\subsubsection{Back to the lower bound of  $\esp{\Tilde{J}_t + \Tilde{J}_{t+1} | f_t = (1,0,0,0)}$.}
By putting the lower bound $I_0$ in the Equation~\eqref{eq:driftAux1000}, see Appendix~\ref{sec:lb2n31000} for details, we find that
\begin{equation} \label{eq:-21000}
\esp{\Tilde{J}_t + \Tilde{J}_{t+1} | f_t = (1,0,0,0)} > -2.
\end{equation}
To conclude the proof,  we now prove that the same inequality holds for $f_t = (0,0,0,0)$.

\subsubsection{Lower bound of $\esp{\Tilde{J}_t + \Tilde{J}_{t+1} | f_t = (0,0,0,0)}$.}
Using Equation~\eqref{eq:trans0000}, we obtain
\begin{align*}
& \esp{\Tilde{J}_{t} + \Tilde{J}_{t+1} |f_t=(0,0,0,0)}\\
&= \esp{\Tilde{J}_{t}|f_t=(0,0,0,0)} + \epsilon_1 \epsilon_0^2 r^2 \esp{\Tilde{J}_{t+1}|f_{t+1}=(1,*,0,0)}\\
&+ \dots + \left(\sum_{l=0}^2 \sum_{k=0}^\infty  (1-r)^{k+3-l} \epsilon_0^{l} \epsilon_1 \epsilon_0^{l} r \right) \esp{\Tilde{J}_{t+1} | f_{t+1} \in S_1}. 
\end{align*}
Finally, there are three terms that are not known and that we have to bound by below. For shortness, we denote by $E^*$ the set $\{0,*\}\times\{*\}\times\{0\}^2$. By considering all the possible types of boundary in each subset, similarly as in Equation~\eqref{eq:inf*0}, we obtain
\begin{align*}
&\esp{\Tilde{J_t}|f_t=(*,0,0,0)}\geq \min \left( \esp{\Tilde{J_t}|f_t=(0,0,0,0)} ,\esp{\Tilde{J_t}|f_t=(1,0,0,0)}\right);\\
&\esp{\Tilde{J_t}|f_t=(1,*,0,0)}\geq \min \left( \esp{\Tilde{J_t}|f_t=(1,0,0,0)},\esp{\Tilde{J_t}|f_t\in S_1} \right);\\
&\esp{\Tilde{J_t}|f_t\in E^*}\geq \min \left( \esp{\Tilde{J_t}|f_t=(0,0,0,0)},\esp{\Tilde{J_t}|f_t=(1,0,0,0)},\esp{\Tilde{J_t}|f_t\in S_1} \right).
\end{align*}
For any $(\epsilon_0,\epsilon_1)\in[0,\frac{1}{2}]^2$, we have, see Appendices~\ref{sec:cn3ft0} and~\ref{sec:cn3ft1}, the two following inequalities:
\begin{displaymath}
    \esp{\Tilde{J_t}|f_t\in S_1} \geq I_0 \text{ and } \esp{\Tilde{J_t}|f_t\in S_1} \geq \esp{\Tilde{J_t}|f_t=(1,0,0,0)}.
\end{displaymath}
Hence,
\begin{align*}
    &\esp{\Tilde{J_t}|f_t=(*,0,0,0)}\geq\min \left( I_0 ,\esp{\Tilde{J_t}|f_t=(1,0,0,0)}\right);\\
    &\esp{\Tilde{J_t}|f_t= (1,*,0,0)} \geq \esp{\Tilde{J_t}|f_t=(1,0,0,0)};\\
    &\esp{\Tilde{J_t}|f_t\in E^*}\geq \min \left( I_0,\esp{\Tilde{J_t}|f_t=(1,0,0,0)}\right).
\end{align*}
Finally, we get
\begin{align}
& \esp{\Tilde{J}_{t}+\Tilde{J}_{t+1} | f_t=(0,0,0,0)}\nonumber \\
& \geq I_0 \nonumber + \prob{f_{t+1}=(0,0,0,0) |f_t=(0,0,0,0)} I_0 \nonumber \\
& + \prob{f_{t+1} \in S_1 | f_t = (0,0,0,0)} \esp{\Tilde{J}_{t+1}| f_{t+1} \in S_1} \nonumber \\
& + \prob{f_{t+1} \in \{(1,0,0,0),(1,*,0,0)\} |f_t=(0,0,0,0)} \esp{\Tilde{J}_{t+1}|f_{t+1}=(1,0,0,0)} \nonumber\\
& + \prob{f_{t+1}=(*,0,0,0) \cup E^* |f_t=(0,0,0,0)} \min \left( I_0, \esp{\Tilde{J}_{t+1}|f_{t+1}=(1,0,0,0)} \right) \label{eq:lb0000}
\end{align}

The value of $\min \left( I_0, \esp{\Tilde{J}_{t+1}|f_{t+1}=(1,0,0,0)} \right)$ can be the left or the right term, depending on the values of $\epsilon_0$ and $\epsilon_1$. Nevertheless, in Appendix~\ref{sec:lb2n30000}, we prove that in both cases it is greater than $-2$, and so
\begin{equation}
    \esp{\Tilde{J}_{t}+\Tilde{J}_{t+1} | f_t=(0,0,0,0)} > -2.
\end{equation}
Combined with Equations~\eqref{eq:driftS1} and~\eqref{eq:-21000}, it implies Equation~\eqref{eq:goal} when $n=3$, and so Theorem~\ref{thm:main} for $n=3$.

\section{Perspectives}
We have proved the desired result for $n=2$ and $n=3$, by examining the behaviour of the boundaries of decorrelated islands. The ergodicity of the hard-core PCA for $n\geq 4$ is still an open question. Our method could be used to handle the case of larger values of $n$, but the computations are becoming increasingly complex with the size of the neighbourhood, so that the $n=4$ case is already difficult to tackle. However, one can ask whether there could be an automated way of handling larger neighbourhoods, and perhaps also other classes of PCA.

Concerning specifically the hard-core PCA, it seems that for fixed $\epsilon_0$, the mean increment in 1 (and 2) time step(s) is an increasing function in $\epsilon_1$. If this is indeed the case, it could simplify the analysis.

\bibliographystyle{alpha}
\newcommand{\etalchar}[1]{$^{#1}$}

\newpage

\section{Proofs of the lower bounds of the drifts} \label{sec:Appendix}
\subsection{Lower bound of $\esp{\Tilde{J}_{t} + \Tilde{J}_{t+1} | f_t = (0,0)}$ for $n=2$} \label{sec:proofn200}
By Equation~\eqref{eq:2pas}, Inequalities~\eqref{eq:inf00}~and~\eqref{eq:inf*0}, and Equality~\eqref{eq:1stepn21}, we obtain the following lower bound: 
\begin{align*}
& \esp{\Tilde{J}_{t}+\Tilde{J}_{t+1} | f_t=(0,0)} \geq  (1+\epsilon_0 r^2 + 2 \epsilon_0^2 r + \epsilon_0^2) \left(-\frac{1}{2} r^2 - 2 \epsilon_0 r + \epsilon_0^2 r + \frac{1}{2} \epsilon_1^2 + \frac{(1-r)^2}{r} \right) \nonumber\\
& + (1-\epsilon_0 r^2 - 2 \epsilon_0^2 r - \epsilon_0^2) \left( -\frac{1}{2}r + \epsilon_1 r + \frac{1}{2} \epsilon_0 (1-\epsilon_1)r + \frac{1}{2} \epsilon_1^2 r + \frac{1}{2} (1-r)^2 + \frac{1}{2}\epsilon_1^2 + \frac{(1-r)^2}{r} \right).
\end{align*}

Using a computer algebra system, we find the following polynomial expression in $\epsilon_0$ and $\epsilon_1$:
\begin{align*}
& \esp{\Tilde{J}_{t}+\Tilde{J}_{t+1} | f_t=(0,0)} \geq -1 + 2 \frac{(1-r)^2}{r} \\
& + \underbrace{\epsilon_0^6 + \frac{3}{2} \epsilon_0^5 \epsilon_1}_{\geq 0} +\underbrace{\frac{7}{2} \epsilon_0^4}_{(1)} + \underbrace{6 \epsilon_0^3 \epsilon_1^2}_{(2)} + \underbrace{\frac{5}{2} \epsilon_0^2 \epsilon_1^3}_{\geq 0} + \underbrace{4 \epsilon_0^2 \epsilon_1}_{(3)} + \underbrace{\frac{3}{2} \epsilon_0^2}_{(4)} + \underbrace{\frac{1}{2} \epsilon_0 \epsilon_1^5}_{\geq 0} + \underbrace{\epsilon_0 \epsilon_1^3 + \epsilon_0 \epsilon_1^2}_{(5)} + \underbrace{\frac{5}{2} \epsilon_1}_{(6)} \\
& - \left(\underbrace{\frac{7}{2} \epsilon_0^5}_{(1)} + \underbrace{\epsilon_0^4 \epsilon_1^2}_{(2)} + \underbrace{\frac{3}{2}\epsilon_0^4 \epsilon_1}_{(6)} + \underbrace{2\epsilon_0^3 \epsilon_1^3}_{(2)} + \underbrace{4\epsilon_0^3 \epsilon_1}_{(3)} + \underbrace{\frac{3}{2} \epsilon_0^3}_{(4)} + \frac{13}{2} \epsilon_0^2 \epsilon_1^2 + \underbrace{\frac{3}{2} \epsilon_0 \epsilon_1^4}_{(5)} + \underbrace{\frac{1}{2} \epsilon_0 \epsilon_1}_{(6)} \right).
\end{align*}

Now, we regroup terms of the form $c \epsilon_0^{\alpha_0} \epsilon_1^{\alpha_1}$ with terms of the form $-c \epsilon_0^{\beta_0} \epsilon_1^{\beta_1}$, with $\alpha_0 \leq \alpha_1$ and $\beta_0 \leq \beta_1$, in such a way that their sums are
\begin{displaymath}
c \epsilon_0^{\alpha_0} \epsilon_1^{\alpha_1} (1 - \epsilon_0^{\beta_0 - \alpha_0} \epsilon_1^{\beta_1 - \alpha_1}) \geq 0
\end{displaymath}
for any $\epsilon_0, \epsilon_1 \in [0,1]$.

We then obtain
\begin{align*}
\esp{\Tilde{J}_{t}+\Tilde{J}_{t+1} | f_t=(0,0)} \geq -1 + 2 \frac{(1-r)^2}{r} - \frac{13}{2} \epsilon_0^2 \epsilon_1^2+ \underbrace{(6-1-2) \epsilon_0^3 \epsilon_1^2}_{(2)} + \underbrace{\left(1-\frac{1}{2}\right) \epsilon_0 \epsilon_1^2}_{(5)} + \underbrace{\left(\frac{5}{2}-\frac{3}{2}-\frac{1}{2}\right) \epsilon_1}_{(6)}.
\end{align*}
To conclude, remark that 
\begin{align*}
2 \frac{(1-r)^2}{r} & = 2 (1-r)^2 + 2 \frac{(1-r)^3}{r} \geq 2 (\epsilon_0 + \epsilon_1)^2 = 2 \epsilon_0^2 + 4 \epsilon_0 \epsilon_1 + 2 \epsilon_1^2.
\end{align*}
This is enough to compensate the last negative term $\displaystyle -\frac{13}{2} \epsilon_0^2 \epsilon_1^2$.

\subsection{Lower bound of $\esp{\Tilde{J}_{t} + \Tilde{J}_{t+1} | f_t = (1,0,0,0)}$ for $n=3$} \label{sec:lb2n31000}
By Equations~\eqref{eq:driftAux1000}, \eqref{eq:trans1000}, \eqref{eq:driftS1}, \eqref{eq:1step1000} and~\eqref{eq:I0}, we obtain the following lower bound:
\begin{align*}
& \esp{\Tilde{J}_{t} + \Tilde{J}_{t+1}|f_t=(1,0,0,0)}\\
& \geq \esp{\Tilde{J}_{t}|f_t=(1,0,0,0)} + \prob{f_{t+1} \in S_1 |f_t=(1,0,0,0)} \esp{\Tilde{J}_{t+1}|f_{t+1}\in S_1} \nonumber \\
    & + \prob{f_{t+1}=(1,0,0,0) |f_t=(1,0,0,0)} \esp{\Tilde{J}_{t+1}|f_{t+1}=(1,0,0,0)} \nonumber \\
    & + \prob{f_{t+1}=(0,0,0,0) |f_t=(1,0,0,0)} I_0.
\end{align*}

Now, replacing the value in $\epsilon_0$ and $\epsilon_1$ according to~\eqref{eq:trans1000}, \eqref{eq:driftS1}, \eqref{eq:1step1000} and~\eqref{eq:I0}, and with the help of a computer algebra system (see Appendix 7), we obtain
\begin{align*}
& \esp{\Tilde{J}_{t} + \Tilde{J}_{t+1}|f_t=(1,0,0,0)} \geq -2+ 2 \frac{(1-r)^3}{r}\\
& + 3 \epsilon_1 + \underbrace{4 \epsilon_1^2}_{(1)} + \epsilon_1^4 + 2 \epsilon_0 \epsilon_1 + 3 \epsilon_0 \epsilon_1^2 + 5 \epsilon_0 \epsilon_1^3 + \underbrace{14 \epsilon_0 \epsilon_1^5}_{(2)} + \epsilon_0 \epsilon_1^7 + 7 \epsilon_0^2 \epsilon_1 + \underbrace{13 \epsilon_0^2 \epsilon_1^3}_{(3)} \\
& + 10 \epsilon_0^2 \epsilon_1^5 + \epsilon_0^2 \epsilon_1^7 + 11 \epsilon_0^3 \epsilon_1 + \underbrace{26 \epsilon_0^3 \epsilon_1^3}_{(4)} + 2 \epsilon_0^3 \epsilon_1^5 + \underbrace{3 \epsilon_0^4}_{(5)+(6)} + \underbrace{49 \epsilon_0^4 \epsilon_1^2}_{(7)} +18 \epsilon_0^4 \epsilon_1^4\\
& + 12 \epsilon_0^5 \epsilon_1 + 11 \epsilon_0^5 \epsilon_1^3 + 2 \epsilon_0^6 \epsilon_1^2 + \underbrace{3 \epsilon_0^7}_{(6)} + 3 \epsilon_0^8 \epsilon_1 + 2 \epsilon_0^9 \\
& - (\underbrace{3 \epsilon_1^3}_{(1)} + 15 \epsilon_0 \epsilon_1^4 + \underbrace{6 \epsilon_0 \epsilon_1^6}_{(2)} +13 \epsilon_0^2 \epsilon_1^2 + \underbrace{12 \epsilon_0^2 \epsilon_1^4}_{(3)} + \underbrace{5 \epsilon_0^2 \epsilon_1^6}_{(2)} + 25 \epsilon_0^3 \epsilon_1^2 +\underbrace{12 \epsilon_0^3 \epsilon_1^4}_{(4)} + 24 \epsilon_0^4 \epsilon_1\\
& +\underbrace{43 \epsilon_0^4 \epsilon_1^3}_{(7)} + \underbrace{3 \epsilon_0^4 \epsilon_1^5}_{(4)} + \underbrace{\epsilon_0^5}_{(5)} + 19 \epsilon_0^5 \epsilon_1^2 + \underbrace{2 \epsilon_0^5 \epsilon_1^4}_{(5)} + \epsilon_0^6 \epsilon_1 + \epsilon_0^6 \epsilon_1^3 + 4 \epsilon_0^7 \epsilon_1 + \underbrace{5 \epsilon_0^8}_{(6)} )
\end{align*}

As in Appendix~\ref{sec:proofn200}, we regroup terms of the form $c \epsilon_0^{\alpha_0} \epsilon_1^{\alpha_1}$ with terms of the form $-c \epsilon_0^{\beta_0} \epsilon_1^{\beta_1}$, with $\alpha_0 \leq \alpha_1$ and $\beta_0 \leq \beta_1$. The last negative terms are then compensated by terms in 
\begin{displaymath}
 \frac{(1-r)^3}{r} = (\epsilon_0+\epsilon_1)^3 + \dots + (\epsilon_0+\epsilon_1)^8 + \frac{(1-r)^9}{r}.
\end{displaymath}
Thus, $\esp{\Tilde{J}_{t} + \Tilde{J}_{t+1}|f_t=(1,0,0,0)} > -2$.

\subsection{Proof that $\esp{\Tilde{J}_t | f_t \in S_1 } \geq I_0$ for $n=3$} \label{sec:cn3ft0}
Equations~\eqref{eq:driftS1} and~\eqref{eq:I0} give that
\begin{align*}
& \esp{\Tilde{J}_t | f_t \in S_1 } - I_0 \\
& = \epsilon_0+2\epsilon_1 + \epsilon_0^2 + 3 \epsilon_1^2 - 3 \epsilon_1^3 + \epsilon_0 \epsilon_1 - \epsilon_0^2 \epsilon_1 + \epsilon_1^4 + \epsilon_0^3 \epsilon_1 \\
& - \left(-\epsilon_0 + \epsilon_1 - \epsilon_0^2 + \epsilon_1^2 + 3 \epsilon_0 \epsilon_1 + 6 \epsilon_0^3 + 2 \epsilon_0^2 \epsilon_1 - 3 \epsilon_0^4 - 3 \epsilon_0^3 \epsilon_1 \right)\\
& = 2 \epsilon_0 + \epsilon_1 + \underbrace{2 \epsilon_0^2 + 2 \epsilon_1^2} + \epsilon_1^4 + 4 \epsilon_0^3 \epsilon_1 + 3 \epsilon_0^4
- (3 \epsilon_1^3 + \underbrace{2 \epsilon_0 \epsilon_1} + 3 \epsilon_0^2 \epsilon_1 + 6\epsilon_0^3 )
\end{align*}

We use a remarkable identity to bound by below the underbraced terms. For the others negative terms, we use the fact that $0 \leq \epsilon_0, \epsilon_1 \leq 1/2$. Hence, 
\begin{displaymath}
3 \epsilon_1^3 \leq 3 \left( \frac{1}{2} \right)^2 \epsilon_1, \ 3 \epsilon_0^2 \epsilon_1 \leq \frac{3}{4} \epsilon_0, \text{ and } 6 \epsilon_0^3 = 4 \epsilon_0^3 + 2 \epsilon_0^3 \leq \epsilon_0 + \epsilon_0^2.
\end{displaymath}

\begin{align*}
& \esp{\Tilde{J}_t | f_t \in S_1 } - I_0 \\
& \geq 2 \epsilon_0 + \epsilon_1 +  \underbrace{\epsilon_0^2 + \epsilon_1^2 +(\epsilon_0-\epsilon_1)^2} + \epsilon_1^4 + 4 \epsilon_0^3 \epsilon_1 + 3 \epsilon_0^4
- \left(\frac{3}{4} \epsilon_1 + \frac{3}{4} \epsilon_0 + \epsilon_0+\epsilon_0^2  \right)\\
& = \frac{1}{4}\epsilon_0 + \frac{1}{4}\epsilon_1 + \epsilon_1^2 +(\epsilon_0-\epsilon_1)^2 + \epsilon_1^4 + 4 \epsilon_0^3 \epsilon_1 + 3 \epsilon_0^4 >0.
\end{align*}

\subsection{Proof that $\esp{\Tilde{J}_t | f_t \in S_1 } \geq \esp{\Tilde{J}_t | f_t = (1,0,0,0)}$ for $n=3$} 
\label{sec:cn3ft1}
Equations~\eqref{eq:driftS1} and~\eqref{eq:1step1000} give that
\begin{align*}
& \esp{\Tilde{J}_t | f_t \in S_1 } - \esp{\Tilde{J}_t | f_t = (1,0,0,0)}\\
& = \epsilon_0+2\epsilon_1 + \epsilon_0^2 + 3 \epsilon_1^2 - 3 \epsilon_1^3 + \epsilon_0 \epsilon_1 - \epsilon_0^2 \epsilon_1 + \epsilon_1^4 + \epsilon_0^3 \epsilon_1 \\
& - (-\epsilon_0 + \epsilon_1 + \epsilon_0^2 + \epsilon_1^2 + 2 \epsilon_0 \epsilon_1 + 2 \epsilon_0^3 + 2 \epsilon_0 \epsilon_1^2 - \epsilon_0^4 - \epsilon_0 \epsilon_1^3)\\
& = 2 \epsilon_0 + \epsilon_1 + 2 \epsilon_1^2 + \epsilon_1^4 + \epsilon_0^3 \epsilon_1 + \epsilon_0^4 + \epsilon_0 \epsilon_1^3\\
& - (3 \epsilon_1^3 +\epsilon_0\epsilon_1 + \epsilon_0^2 \epsilon_1 + 2 \epsilon_0^3 + 2 \epsilon_0 \epsilon_1^2 ).
\end{align*}

As in Appendix~\ref{sec:cn3ft0}, the fact that $0 \leq \epsilon_0, \epsilon_1 \leq 1/2$ allows us to bound by below the negative terms.
\begin{align*}
& \esp{\Tilde{J}_t | f_t \in S_1 } - \esp{\Tilde{J}_t | f_t = (1,0,0,0)}\\
& \geq 2 \epsilon_0 + \epsilon_1 + 2 \epsilon_1^2 + \epsilon_1^4 + \epsilon_0^3 \epsilon_1 + \epsilon_0^4 + \epsilon_0 \epsilon_1^3 - \left(\frac{3}{4} \epsilon_1 +\frac{1}{2}\epsilon_0 + \frac{1}{4}\epsilon_1 + \frac{1}{2} \epsilon_0 + \frac{1}{2} \epsilon_0 \right)\\
& =  \frac{1}{2} \epsilon_0 + 2 \epsilon_1^2 + \epsilon_1^4 + \epsilon_0^3 \epsilon_1 + \epsilon_0^4 + \epsilon_0 \epsilon_1^3 > 0.
\end{align*}

\subsection{Lower bounds of $\esp{\Tilde{J}_{t} + \Tilde{J}_{t+1} | f_t = (0,0,0,0)}$ for $n=3$} \label{sec:lb2n30000}
\subsubsection{When the minimum is $I_0$:}
The lower bound of Equation~\eqref{eq:lb0000} becomes
\begin{align*}
D_0 & = I_0 \nonumber + \prob{f_{t+1}=(0,0,0,0) |f_t=(0,0,0,0)} I_0 \nonumber \\
& + \prob{f_{t+1} \in S_1 | f_t = (0,0,0,0)} \esp{\Tilde{J}_{t+1}| f_{t+1} \in S_1} \nonumber \\
& + \prob{f_{t+1} \in \{(1,0,0,0),(1,*,0,0)\} |f_t=(0,0,0,0)} \esp{\Tilde{J}_{t+1}|f_{t+1}=(1,0,0,0)} \nonumber\\
& + \prob{f_{t+1}=(*,0,0,0) \cup E^* |f_t=(0,0,0,0)} I_0
\end{align*}

Now, replace the terms of this equation by their values given in Equations~\eqref{eq:I0},~\eqref{eq:driftS1},~\eqref{eq:1step1000} and deduced from Equation~\eqref{eq:trans0000}. With the help of a computer algebra system (see Appendix 7), we obtain
\begin{align*}
& D_0 = -2 + 2 \frac{(1-r)^3}{r}\\
& + 3 \epsilon_1 + 4 \epsilon_1^2 +\epsilon_1^4 + 4 \epsilon_0 \epsilon_1 + 6 \epsilon_0^2 \epsilon_1^3 + 5 \epsilon_0^2 \epsilon_1^5  + 4 \epsilon_0^3 + 2 \epsilon_0^3 \epsilon_1 + 14 \epsilon_0^3 \epsilon_1^3 + 5 \epsilon_0^3 \epsilon_1^5 +15 \epsilon_0^4 \epsilon_1 \\
& + \epsilon_0^4 \epsilon_1^4  + 10 \epsilon_0^5 + 15 \epsilon_0^5 \epsilon_1^2 + 2 \epsilon_0^5 \epsilon_1^4 + 6 \epsilon_0^6 \epsilon_1 + 6 \epsilon_0^7 + 3 \epsilon_0^7 \epsilon_1^2 + 8 \epsilon_0^8 \epsilon_1 + 4 \epsilon_0^9\\
& - (3\epsilon_1^3 + 9 \epsilon_0^2 \epsilon_1^4 + \epsilon_0^2 \epsilon_1^6 + 10 \epsilon_0^3 \epsilon_1^2 + 11 \epsilon_0^3 \epsilon_1^4 + \epsilon_0^3 \epsilon_1^6 + 9 \epsilon_0^4  + 7 \epsilon_0^4 \epsilon_1^2 + \epsilon_0^4 \epsilon_1^3 + 18 \epsilon_0^5 \epsilon_1 \\
& + 8 \epsilon_0^5 \epsilon_1^3 + 3 \epsilon_0^6 + 4 \epsilon_0^6 \epsilon_1^2 + 11 \epsilon_0^7 \epsilon_1  + 10 \epsilon_0^8) 
\end{align*}

We conclude that it is greater than $-2$ as in Appendices~\ref{sec:proofn200} and~\ref{sec:lb2n31000}, by grouping terms and developing $(1-r)^3/r$ as much as necessary. 

This technique does not hold for the term in $\epsilon_0^4$. Indeed, we obtain $6 \epsilon_0^3 - 7 \epsilon_0^4$ after that. To prove its positivity, we use the fact that $\epsilon_0 \leq 1/2$, and so $6 \epsilon_0^3 - 7 \epsilon_0^4 \geq (6-7/2) \epsilon_0^3 > 0$.

\subsubsection{When the minimum is $\esp{\Tilde{J}_{t+1}|f_{t+1}=(1,0,0,0)}$:}
The lower bound of Equation~\eqref{eq:lb0000} becomes
\begin{align*}
D_1 & = I_0 + \prob{f_{t+1}=(0,0,0,0) |f_t=(0,0,0,0)} I_0 \nonumber \\
& + \prob{f_{t+1} \in S_1 | f_t = (0,0,0,0)} \esp{\Tilde{J}_{t+1}| f_{t+1} \in S_1} \nonumber \\
& + \prob{f_{t+1} \in \{(1,0,0,0),(1,*,0,0)\} |f_t=(0,0,0,0)} \esp{\Tilde{J}_{t+1}|f_{t+1}=(1,0,0,0)} \nonumber\\
& + \prob{f_{t+1}=(*,0,0,0) \cup E^* |f_t=(0,0,0,0)} \esp{\Tilde{J}_{t+1}|f_{t+1}=(1,0,0,0)}
\end{align*}

Now, replace the terms of this equation by their values given in Equations~\eqref{eq:I0},~\eqref{eq:driftS1},~\eqref{eq:1step1000} and deduced from Equation~\eqref{eq:trans0000}. With the help of a computer algebra system (see Appendix 7), we obtain
\begin{align*}
& D_1 = -2 + 2 \frac{(1-r)^3}{r}\\
& + 3 \epsilon_1 + 4 \epsilon_1^2 + \epsilon_1^4 + 4 \epsilon_0 \epsilon_1 + 6 \epsilon_0^2 \epsilon_1^3  + 5 \epsilon_0^2 \epsilon_1^5 + 4 \epsilon_0^3 + \epsilon_0^3 \epsilon_1 + 4 \epsilon_0^3 \epsilon_1^3 + 7 \epsilon_0^4 \epsilon_1 + 5 \epsilon_0^4 \epsilon_1^2 \\
& + 3 \epsilon_0^4 \epsilon_1^4 + 6 \epsilon_0^5 + 14 \epsilon_0^5 \epsilon_1^3 + 10 \epsilon_0^6 \epsilon_1 + 22 \epsilon_0^7 + 12 \epsilon_0^7 \epsilon_1^2 +20 \epsilon_0^8 \epsilon_1 + 8 \epsilon_0^9 \\
& - (3 \epsilon_1^3 + 9 \epsilon_0^2 \epsilon_1^4 + \epsilon_0^2 \epsilon_1^6  + 5 \epsilon_0^3 \epsilon_1^2 + \epsilon_0^3 \epsilon_1^4 + 7 \epsilon_0^4 + 9 \epsilon_0^4 \epsilon_1^3 \\
& + 15 \epsilon_0^5 \epsilon_1^2 + 4 \epsilon_0^5 \epsilon_1^4 + 7 \epsilon_0^6 + 4 \epsilon_0^6 \epsilon_1^3 + 36 \epsilon_0^7 \epsilon_1 + 24 \epsilon_0^8).
\end{align*}

We conclude that it is greater than $-2$ as in Appendices~\ref{sec:proofn200} and~\ref{sec:lb2n31000}, by grouping terms and developing $(1-r)^3/r$ as much as necessary.

\section{Code}

In order to obtain the lower bounds for the drift, we have implemented the calculations both in Python and in Mathematica. We present our codes in the next two subsections.
\subsection{Python code}

Below is the Python code we used to perform the drift calculations. 

\begin{lstlisting}[language=Python]
from sympy import *
import numpy as np

x, y = symbols('\u03B5_0 \u03B5_1')    ## Epsilon symbols

init_printing(use_unicode=True)

r=1-x-y

#For n = 3

#%%Drift on 1 step for the boundary S1 (eq. 12)

DS1 = np.sum([(-1+k) * (1-y)**(3-k) * x**k * r for k in [0,1,2]])
DS1 += (1 + 1/r) * x**3
for l in [0,1,2]:
    DS1 += np.sum([(-l+k) * y * (1-y)**(l-k) * x**k * r for k in range(0,l+1)])
    if l != 2:
        DS1 += np.sum([(-l+k) * (x+y)**(k-l-1) * y * x**l * r for k in range(l+1,3)])
    DS1 += ((2-l) + 1/r) * (x+y)**(2-l) * x**l * y 

print(simplify(DS1 - (1-r)**3/r))

#%%Drift on 1 step for the boundary 1000 (eq. 15)

D10 = np.sum([(k-2) * (1-y)**(3-k) * x**(1+k) * r for k in [0,1,2,3]])
D10 += np.sum([(k-2) * y * (1-y)**(2-k) * x**(1+k) * r for k in [0,1,2]])
D10 += (1-x) * x**3 * r
for l in [0,1,2]:
    if l!=0:
        D10 += np.sum([(k-l-1) * x**(k+1) * y * (1-y)**(l-k-1) * r for k in range(0,l)])
    D10 += (-1) * (1-x) * x**l * r
    D10 += np.sum([(k-l-1) * (x+y)**(k-l-1) * y * x**l * r for k in range(l+1,4)])
    D10 += (2-l+1/r) * ((x+y)**(3-l) * y * x**l)
D10 += (1+1/r) * x**4
D10 += (1+1/r) * y * x**3

print(simplify(D10-(x+y)**3/r))

#%%Drift on 1 step for the boundary 0000 (eq. 17)

D00 = -3 * y * x**2 * r**2
D00 += -3 * (1-y) * x**2 * r**2
D00 += np.sum([(k-2) * x**3 * r**2 for k in [0,1]])
D00 += -3 * x**2 * y * r
D00 += -2 * (1-x) * x**2 * r
D00 += np.sum([(k-2) * x**4 * r for k in [0,1,2,3]])
D00 += np.sum([(k-2) * y * x**3 * r for k in [0,1]])
D00 += np.sum([(k-2) * (1-x) * x**3 * r for k in [2,3]])
D00 += -(1-x) * r
D00 += 0
D00 += (-2) * (1-x) * x * r
D00 += (-1) * (1-x) * x * r
D00 += np.sum([(k+1) * (x+y)**(k+1) * y * r for k in [0,1]])
D00 += np.sum([(k) * (x+y)**k * y * x * r for k in [0,1]])
D00 += (-1) * (1-x) * x**2 * r
D00 += 0
D00 += (1+1/r) * x**4
D00 += (1+1/r) * x**3 * y
for l in [0,1,2]:
    D00 += (2-l+1/r) * (x+y)**(3-l) * y * x**l
    
print(simplify(D00-(x+y)**3/r))

#%%Transition probabilities of type P(y_{t+1} in E | y_t in F) 

#y_t = 1000 (eq. 13)

#y_t+1 = S1 
PS1_10 = 0 
for l in [0,1,2]:
    if l!=0:
        PS1_10 += np.sum([x**(k+1) * y * (1-y)**(l-k-1) * r for k in range(0,l)])
    PS1_10 += (1-x) * x**l * r
    PS1_10 += np.sum([(x+y)**(k-l-1) * y * x**l * r for k in range(l+1,4)])
    PS1_10 += (x+y)**(3-l) * y * x**l

#y_t+1 = 1000
P10_10 = np.sum([y * (1-y)**(2-k) * x**(1+k) * r for k in [0,1,2]])
P10_10 += (1-x) * x**3 * r
P10_10 += y * x**3

#y_t+1 = 0000
P00_10 = np.sum([(1-y)**(3-k) * x**(1+k) * r for k in [0,1,2,3]])
P00_10 += x**4

#Verify if the sum is equal to 1
print(simplify(PS1_10+P10_10+P00_10))

#%%
#y_t = 0000 (eq. 16)

#y_t+1 = S1 
PS1_00 = x**2 * y * r
PS1_00 += (1-x) * x**2 * r
PS1_00 += (1-x) * r
PS1_00 += y * r
PS1_00 += (1-x) * x * r
PS1_00 += (1-x) * x * r
PS1_00 += np.sum([(x+y)**(k+1) * y * r for k in [0,1]])
PS1_00 += np.sum([(x+y)**k * y * x * r for k in [0,1]])
PS1_00 += (1-x) * x**2 * r
PS1_00 += y * x**2 * r
for l in [0,1,2]:
    PS1_00 += (x+y)**(3-l) * y * x**l

#y_t+1 = 1000
P10_00 = np.sum([y * x**3 * r for k in [0,1]])
P10_00 += np.sum([(1-x) * x**3 * r for k in [2,3]])
P10_00 += x**3 * y

#y_t+1 = 0000
P00_00 = np.sum([x**4 * r for k in [0,1,2,3]])
P00_00 += x**4

#y_t+1 = *000
PE0_00 = np.sum([x**3 * r**2 for k in [0,1]])

#y_t+1 = E*
PEE_00 = (1-y) * x**2 * r**2

#y_t+1 = 1*00
P1E_00 = y * x**2 * r**2

#Verify if the sum is equal to 1
print(simplify(PS1_00+P10_00+P00_00+PE0_00+PEE_00+P1E_00))

#%%Drift on 2 steps for the boundary 1000 (sec. 6.2)

D2T_10 = D10+PS1_10 * DS1+P10_10 * D10+P00_10 * D00
print(simplify(D2T_10-2 * (1-r)**3/r))

#%%Drift on 2 steps for the boundary 0000 (sec. 6.5)

#%%Case where we major the term depending on E^* by I0 (sec. 6.5.1)
D2T_00_0 = D00+P00_00 * D00+PS1_00 * DS1+(P10_00+P1E_00) * D10+(PE0_00+PEE_00) * D00
print(simplify(D2T_00_0-2*(1-r)**3/r))

#%%Verification of the majoration by -2 by modifying 2*(1-r)**3/r
bino_0 = np.sum([(x+y)**k for k in range(3,10)])
print(simplify(D2T_00_0-2*(1-r)**3/r+bino_0))

#%%Case where we major the term depending on E^* by E(J_{t+1}|f_{t+1}=(1000)) (sec. 6.5.2)
D2T_00_1 = D00+P00_00 * D00+PS1_00 * DS1+(P10_00+P1E_00) * D10+(PE0_00+PEE_00) * D10
print(simplify(D2T_00_1 - 2*(1-r)**3/r))

#%%Verification of the majoration by -2 by modifying 2*(1-r)**3/r
bino_1 = np.sum([(x+y)**k for k in range(3,10)])
print(simplify(D2T_00_1-2 * (1-r)**3/r+bino_1))
\end{lstlisting}

\subsection{Mathematica code}

\lstset{language=Mathematica}
\lstset{basicstyle={\sffamily\footnotesize},
  numbers=none,
  breaklines=true,
  captionpos={t},
  frame={lines},
  rulecolor=\color{black},
  framerule=0.5pt,
  columns=flexible,
  tabsize=2
}

Below is the Mathematica code we used to perform the drift calculations.\\

\lstset{language= Mathematica}
\lstset{basicstyle= {\sffamily\footnotesize}, 
  numbers= none, 
  breaklines= true, 
  captionpos= {t}, 
  frame= {lines}, 
  rulecolor= \color{black}, 
  framerule= 0.5pt, 
  columns= flexible, 
  tabsize= 2
}

Input: 
\begin{lstlisting}[language= Mathematica]
	 r := 1-x-y;
\end{lstlisting}
\subsection*{Section 4.2}

\subsection*{Drift on one step for the boundary $S_1$  (eq. \ref{eq:driftS1})}

Input: 
\begin{lstlisting}[language= Mathematica]
	A1 := Sum[(-1+k)*(1-y)^(3-k)*x^k*r, {k, 0, 2}]
	A2 := Sum[(2+k)*(1-r)^k*x^3*r, {k, 0, Infinity}]
	A3 := Sum[(-l+k)*y*(1-y)^(l-k)*x^k*(1-x-y), {l, 0, 2}, {k, 0, l}]
	A4 := Sum[(-l+k)*(1-r)^(k-l-1)*y*x^l*r, {l, 0, 1}, {k, l+1, 2}]
	A5 := Sum[(3-l+k)*(1-r)^(k+2-l)*y*x^l*r, {l, 0, 2}, {k, 0, Infinity}]
	T12b := Simplify[A1+A2+A3+A4+A5-(1-r)^3/r]
	T12b
\end{lstlisting}
Output:
\begin{lstlisting}[language= Mathematica]
	-1+x+x^2+2y+xy-x^2 x^1+x^3y+3y^2-3y^3+y^4
\end{lstlisting}

\subsection*{Drift on one step for the boundary $(1, 0, 0, 0)$  (eq. \ref{eq:1step1000})}

Input: 
\begin{lstlisting}[language= Mathematica]
	B1 := Sum[(k-2)*(1-y)^(3-k)*x^(1+k)*r, {k, 0, 3}] 
	B2 := Sum[(k-2)*y*(1-y)^(2-k)*x^(1+k)*r, {k, 0, 2}] 
	B3 := (1-x)*x^3*r 
	B4 := Sum[(-1+k-l)*x^(k+1)*y*(1-y)^(l-k-1)*r, {l, 1, 2}, {k, 0, l-1}] 
	B5 := Sum[-(1-x)*x^l*r, {l, 0, 2}] 
	B6 := Sum[(-1+k-l)*(1-r)^(k-l-1)*y*x^l*r, {l, 0, 2}, {k, l+1, 3}] 
	B7 := Sum[(2+k)*(1-r)^k*x^4*r, {k, 0, Infinity}]
	B8 := Sum[(2+k)*(1-r)^k*y*x^3*r, {k, 0, Infinity}] 
	B9 := Sum[(3+k-l)*(1-r)^(3+k-l)*y*x^l*r, {l, 0, 2}, {k, 0, Infinity}]
	T15b := Expand[Simplify[B1+B2+B3+B4+B5+B6+B7+B8+B9-(1-r)^3/r]] 
	T15b
\end{lstlisting}
Output:
\begin{lstlisting}[language= Mathematica]
	-1-x+x^2+2x^3-x^4+y+2xy+y^2+2xy^2-xy^3
\end{lstlisting}

\subsection*{Drift on one step for the boundary $(0, 0, 0, 0)$ (eq. \ref{eq:I0})}

Input: 
\begin{lstlisting}[language= Mathematica]
	C1 := -3*y*r*x^2*r
	C2 := -3*(1-y)*r*x^2*r
	C3 := Sum[(-2+k)*r*x^3*r, {k, 0, 1}]
	C4 := -3*y*x^2*r
	C5 := -2*(1-x)*x^2*r
	C6 := Sum[(-2+k)*x^4*r, {k, 0, 3}]
	C7 := Sum[(-2+k)y*x^3*r, {k, 0, 1}]
	C8 := Sum[(-2+k)*(1-x)*x^3*r, {k, 2, 3}]
	C9 := -1*(1-x)*r
	C10 := -2*(1-x)*x*r
	C11 := -(1-x)*x*r
	C12 := Sum[(1+k)*(1-r)^(k+1)*y*r, {k, 0, 1}]
	C13 := Sum[k*(1-r)^k*y*x*r, {k, 0, 1}]
	C14 := -(1-x)*x^2*r
	C15 := Sum[(2+k)*(1-r)^k*x^4*r, {k, 0, Infinity}]
	C16 := Sum[(2+k)*(1-r)^k*y*x^3*r, {k, 0, Infinity}]
	C17 := Sum[(3-l+k)*(1-r)^(k+3-l)*y*x^l*r, {l, 0, 2}, {k, 0, Infinity}]
	T17b := Expand[Simplify[
		C1+C2+C3+C4+C5+C6+C7+C8+C9+C10+C11+C12+C13+C14+C15+C16+C17-(1-r)^3/r]]
	T17b
\end{lstlisting}
Output:
\begin{lstlisting}[language= Mathematica]
	-1-x-x^2+6x^3-3x^4+y+3xy+2x^2y-3x^3y+y^2
\end{lstlisting}

\subsection*{Section 6.2}

\subsubsection*{Calculation of $\mathbb{P}(f_{t+1} \in S_1 | f_t = (1, 0, 0, 0))$}

Input:
\begin{lstlisting}[language= Mathematica]
	P11 := Sum[x^(k+1)*y*(1-y)^(l-k-1)*r, {l, 1, 2}, {k, 0, l-1}]
	P12 := Sum[(1-x)*x^l*r, {l, 0, 2}] 
	P13 := Sum[(1-(1-x-y))^(k-l-1)*y*x^l*(1-x-y), {l, 0, 2}, {k, l+1, 3}]
	P14 := Sum[(1-(1-x-y))^(3+k-l)*y*x^l*(1-x-y), {l, 0, 2}, {k, 0, Infinity}] 
	P1 := P11+P12+P13+P14
\end{lstlisting}

\subsubsection*{Calculation of $\mathbb{P}(f_{t+1} \in (1, 0, 0, 0) | f_t = (1, 0, 0, 0))$}

Input:
\begin{lstlisting}[language= Mathematica]
	P21 := Sum[y*(1-y)^(2-k)*x^(1+k)*r, {k, 0, 2}]
	P22 := Sum[(1-(1-x-y))^k*y*x^3*(1-x-y), {k, 0, Infinity}] 
	P23 := (1-x)*x^3*(1-x-y)
	P2 := P21+P22+P23
\end{lstlisting}

\subsubsection*{Calculation of $\mathbb{P}(f_{t+1} \in (0, 0, 0, 0) | f_t = (1, 0, 0, 0))$}

Input:
\begin{lstlisting}[language= Mathematica]
	P31 := Sum[(1-y)^(3-k)*x^(1+k)*(1-x-y), {k, 0, 3}]
	P32 := Sum[(1-(1-x-y))^k*x^4*(1-x-y), {k, 0, Infinity}] 
	P3 := P31+P32
\end{lstlisting}

\subsubsection*{Check and verify}

Input:
\begin{lstlisting}[language= Mathematica]
	Simplify[P1+P2+P3]
\end{lstlisting}

Output:
\begin{lstlisting}[language= Mathematica]
	1	
\end{lstlisting}

Input:
\begin{lstlisting}[language= Mathematica]
	Expand[Simplify[T15+T12*P1+T15*P2+T17*P3-2*(1-r)^3/r]]
\end{lstlisting}

Output:
\begin{lstlisting}[language= Mathematica]
-2+3x^4-x^5+3x^7-5x^8+2x^9+3y+2xy+7x^2y+11x^3y-24x^4y+12x^5y-x^6y-4x^7y
+3x^8y+4y^2+3xy^2-13x^2y^2-25x^3y^2+49x^4y^2-19x^5y^2+2x^6y^2-3y^3+5xy^3
+13x^2y^3+26x^3y^3-43x^4y^3+11x^5y^3-x^6y^3+y^4-15xy^4-12x^2y^4-12x^3y^4
+18x^4y^4-2x^5y^4+14xy^5+10x^2y^5+2x^3y^5-3x^4y^5-6xy^6-5x^2y^6+xy^7+x2y^7
\end{lstlisting}

\subsection*{Section 6.5}

\subsubsection*{Calculation of $\mathbb{P}(f_{t+1} \in S_1 | f_t = (0, 0, 0, 0))$}

Input:
\begin{lstlisting}[language= Mathematica]
	Q11 := y*x^2*r+(1-x)*x^2*r+(1-x)*r+
		y*r+(1-x)*x*r+(1-x)*x*r+(1-x)*x^2*r+y*x^2*r
	Q12 := Sum[(1-r)^(k+1)*y*r, {k, 0, 1}]
	Q13 := Sum[(1-r)^k*y*x*r, {k, 0, 1}]
	Q14 := Sum[(1-r)^(3+k-l)*y*x^l*r, {l, 0, 2}, {k, 0, Infinity}]
	Q1 := Q11+Q12+Q13+Q14
\end{lstlisting}

\subsubsection*{Calculation of $\mathbb{P}(f_{t+1} \in \{(1, 0, 0, 0), (1, *, 0, 0)\} | f_t = (0, 0, 0, 0))$}

Input:
\begin{lstlisting}[language= Mathematica]
	Q21 := 2*y*x^3*r+2*(1-x)*x^3*r
	Q22 := Sum[(1-r)^k*y*x^3*r, {k, 0, Infinity}]
	Q23 := y*r*x^2*r
	Q2 := Q21+Q22+Q23
\end{lstlisting}

\subsubsection*{Calculation of $\mathbb{P}(f_{t+1} \in (0, 0, 0, 0) | f_t = (0, 0, 0, 0))$}

Input:
\begin{lstlisting}[language= Mathematica]
	Q31 := 4*x^4*r 
	Q32 := Sum[(1-r)^k*x^4*r, {k, 0, Infinity}]
	Q3 := Q31+Q32
\end{lstlisting}

\subsubsection*{Calculation of $\mathbb{P}(f_{t+1} \in (*, 0, 0, 0)\cup E^*| f_t = (0, 0, 0, 0))$}

Input:
\begin{lstlisting}[language= Mathematica]
	Q4 := (1-y)*r*x^2*r+2*r*x^3*r
\end{lstlisting}

\subsubsection*{Check and verify}

Input:
\begin{lstlisting}[language= Mathematica]
	Simplify[Q1+Q2+Q3+Q4]
\end{lstlisting}

Output:
\begin{lstlisting}[language= Mathematica]
	1	
\end{lstlisting}

\subsubsection*{When the min is $I_0$}

Input:
\begin{lstlisting}[language= Mathematica]
	Expand[Simplify[T17+T12*Q1+T15*Q2+T17*Q3+T17*Q4-2*(1-r)^3/r]]
\end{lstlisting}

Output:
\begin{lstlisting}[language= Mathematica]
	-2+4x^3-9x^4+10x^5-3x^6+6x^7-10x^8+4x^9+3y+4xy+2x^3y+15x^4y-18x^5y+6x^6y
	-11x^7y+8x^8y+4y^2-10x^3y^2-7x^4y^2+15x^5y^2-4x^6y^2+3x^7y^2-3y^3+6x^2y^3
	+14x^3y^3-x^4y^3-8x^5y^3+y^4-9x^2y^4-11x^3y^4+x^4y^4+2x^5y^4+5x^2y^5
	+5x^3y^5-x^2y^6-x^3y^6
\end{lstlisting}

\subsubsection*{When the min is $\esp{J_{t} | f_t = (1, 0, 0, 0)}$}

Input:
\begin{lstlisting}[language= Mathematica]
	Expand[Simplify[T17+T12*Q1+T15*Q2+T17*Q3+T15*Q4-2*(1-r)^3/r]]
\end{lstlisting}

Output:
\begin{lstlisting}[language= Mathematica]
	-2+4x^3-7x^4+6x^5-7x^6+22x^7-24x^8+8x^9+3y+4xy+x^3y+7x^4y+10x^6y-36x^7y
	+20x^8y+4y^2-5x^3y^2+5x^4y^2-15x^5y^2+12x^7y^2-3y^3+6x^2y^3+4x^3y^3
	-9x^4y^3+14x^5y^3-4x^6y^3+y^4-9x^2y^4-x^3y^4+3x^4y^4-4x^5y^4+5x^2y^5
	-x^2y^6
\end{lstlisting}

\end{document}